\documentclass[12pt]{article}%\documentclass[reqno, makeidx]{amsart}

\usepackage{amstext,amssymb,amsmath,stmaryrd,amsthm,mathrsfs,amsfonts,amsbsy,}
\usepackage{geometry} % see geometry.pdf on how to lay out the page.  
\usepackage{color}
\usepackage[nottoc]{tocbibind} 
\usepackage[colorlinks]{hyperref}

\usepackage{latexsym}
\usepackage{graphicx}
\newcommand{\diam}{\mathrm{diam}}

\newcommand{\cX}{\mathcal{X}}

\newtheorem{teo}{Theorem}[section]

\newtheorem{lemma}[teo]{Lemma}
\newtheorem{coroll}[teo]{Corollary}
\newtheorem{rem}[teo]{Remark}
\newtheorem{prop}[teo]{Proposition}
\newtheorem{ex}[teo]{Example}
\begin{document}

\title{Ergodicity of Markov Semigroups with H\"ormander type generators in
Infinite Dimensions
\thanks{
%\footnotenumber{1}\footnotetext[1]{$^*$}\\
{Supported by EPSRC %GR/R90994/01 \& 
EP/D05379X/1}} }

\author{Federica Dragoni$^*$, Vasilis Kontis$^\ddag$, Bogus{\l}aw Zegarli\'nski$^\ddag$
\\
%\\ $\ $\\
%{\small{Department of Mathematics}}\\
%added 07/11
{\small{$^*$ University of Padova}}\\
%{\small{London, UK}}\\
%\\ $\ $\\
%\\
{\small{$^\ddag$Imperial College London}}\\
}

\date{}

\maketitle
%\begin{abstract}

%\end{abstract} 
{\small{
\noindent\textbf{Abstract}: 
{\em We develop an effective strategy for proving strong ergodicity of
(nonsymmetric) Markov semigroups associated to H\"ormander type generators when
the underlying configuration space is infinite dimensional.  }
}}\\

\tableofcontents
\section{Introduction}
After an initial development of a strategy for proving the log-Sobolev inequality
for infinite dimensional H\"ormander type generators $\mathcal{L}$ symmetric in
$L_2(\mu)$ defined with a suitable nonproduct measure $\mu$ (\cite{L-Z},
\cite{H-Z}, \cite{I-P}, \cite{I}), one can envisage an extension of the established
strategy (see e.g. \cite{Z}) for proving strong pointwise
ergodicity for the corresponding Markov semigroups $P_t\equiv e^{t\mathcal{L}}$,
(or in case of the compact spaces even in the uniform norm as in \cite{G-Z} and
references therein).
Still to obtain a fully fledged theory, which could include for example configuration spaces given by general noncompact nilpotent Lie groups other than
Heisenberg type groups, one needs to conquer a (finite dimensional) problem of
sub-Laplacian bounds (of the corresponding control distance). Unfortunately this is a
VP-hard problem which will likely stay with us for more than quite a while. The
other motivation for our work comes also from a desire to get a
strategy for studying Markov semigroups of the above mentioned type which
are not symmetric with respect to some  a priori given reference measure
in cases where the underlying configuration space is infinite dimensional and
noncompact. 
In finite dimensions an interesting analysis in the $L_2$ framework with respect to a reference measure
in particular involving the long time behaviour was provided in \cite{V}. In a number of recent works an interesting progress has been made
in understanding the sub-gradient bounds on finite dimensional sub-Riemaniann manifolds provided
by compact and noncompact Lie groups. Many of the related works (as e.g. \cite{D-M1}, \cite{D-M2}, \cite{M}, \cite{HQL} see also references therein)
are heavily based on complicated stochastic analysis methods with sharp results obtained for Heisenberg type groups. Another insight and complementary understanding were achieved via a more analytic route one can find in \cite{B-B-B-C} and \cite{I-K-Z} (\cite{H-Z}). In particular such bounds involving the length of the sub-gradient offer a nice way of getting smoothing and spectral properties as well as other interesting features coming from related entropy bounds
for the heat kernel. In \cite{B-H-T} an analog of the Orstein-Uhlenbeck processes was proposed and studied
with the drift term provided by the logarithmic derivative of  heat kernels on groups with some general theory involving $L_2$ subgradient bounds and a related Poincar\'e inequality. In \cite{B-T} some stochastic analysis (in a Hilbert space along ideas \cite{DaP-Z}) is studied for certain infinite dimensional models of financial mathematics.
The analysis there concentrates however on hypoellipticity aspects. 
For some other directions involving a hypoellipticity theme in infinite dimensions see e.g. also \cite{B-SC}, \cite{H} and references therein.

In this paper we construct and study Markov semigroups on infinite dimensional
spaces provided for example as an infinite product of noncompact Lie groups (as
e.g. nilpotent free groups), and formulate an effective condition for their
exponential ergodicity in supremum norm. Our main tool is provided by a complete
gradient bound, where the square of the gradient (or
subgradient) is replaced by similar objects but with a family of
fields which is closed with respect to taking commutators with the fields
appearing in the definition of our Markov generator. We assume that our theory
is furnished with some natural dilation operator which when included in the
generator with sufficiently large coefficient assures the exponential dumping.
The use of a complete gradient, while it may not provide us with smoothing
information, it proves to be very effective when the long time behaviour is
concerned, giving also some extra information about the equilibrium measure.
In a finite dimensional setup it provides an alternative view to \cite{B-H-T}.
On the other hand in a general situation when working in infinite dimensions we
have no a priori reference measure and so no natural $L_2$ approach can be used.

The organisation of the paper is as follows. In section 2 we present the general
framework with a number of simple examples, presenting a general idea in finite
dimensions. In section 3 we construct a Markov semigroup in an infinite
dimensional setup proving a  strong approximation property (or as it is
sometimes called a finite speed of propagation of information). This
approximation is later used together with square of the (complete) gradient
bounds to obtain the exponential decay to equilibrium in the supremum norm for a
large class of initial configurations. Finally we conclude with a Poincar\'e type
inequality with complete gradient form which allows us via general arguments
to obtain exponential moments estimates for suitable (generalised) Lipschitz
variables. 

\section{Finite-dimensional case} \label{sec.1}
Consider smooth vector fields $X_1,\dots,X_M$  on $\mathbb{R}^N$, satisfying the H\"ormander condition with step $K>1$.
For \(n\ge N,   \) by $(Z_k)_{k=1}^n$ we denote an adapted family of fields, containing a basis for the related sub-Riemannian geometry. 
So  $Z_k=X_k,$ for $k=1,\dots,M,$ while the remaining $Z_{M+1},\dots, Z_{n}$ are ordered commutators of length   between 2 and 
$K$.\\
For \(m  \le M\), we consider  the following operator:
\begin{equation}
\label{Operator}
\begin{aligned}
&\mathcal{L}:=L+L_G+L_{\alpha}\\
&L:=\sum_{i=1}^m X_i^2-\beta D\\
&L_G:=\sum_{i,j=1}^m G_{i,j} (x)X_iX_j\\
&L_{\alpha}:=\sum_{i=1}^m \alpha_i (x) X_i
\end{aligned}
\end{equation}
where $\beta\in(0,\infty)$ is a constant, $D$ is a first order  dilations generator satisfying
\begin{equation}
\label{Condition_onD}
e^{sD} Z_k e^{-sD}=e^{s\lambda_k}Z_k \textrm{ and }\;[Z_k,D]=\lambda_k Z_k  \textrm{ for some}\;  \lambda_k>0,
\end{equation}
for $k=1,..,n$, and 
$\alpha(x)=(\alpha_1(x),\dots,\alpha_m(x))$ is a smooth function, while  $G(x)=\big(G_{i,j}(x)\big)_{i,j=1}^m$ is an  
$m\times m$-matrix,
satisfying, for any $x\in \mathbb{R}^N$,
\begin{equation}
\label{Condition_onG}
G^*(x)+I>0, \quad \textrm{with}\; G^*_{ij}(x)\equiv \frac12\left(G_{ij}(x)+G_{ji}(x)\right),
\end{equation}
where $I$ is the $m\times m$-identity-matrix. \par Let us introduce the following condition on the geometry of the vector 
fields:
\begin{equation}
\label{CommutatorConstantCoefficents}
\exists \;\; c_{kjl}\in \mathbb{R} \;\; \textrm{such that}\; \;[Z_k,X_j]=\sum_{l=1}^n c_{kjl} Z_l,
\end{equation}
for any $k=1,\dots, n$ and $j=1,\dots ,m$.

\begin{rem}
Note that condition  \eqref{CommutatorConstantCoefficents} is stronger than the H\"ormander condition which implies a similar 
expression, but in general with  non-constant coefficients $c_{kjl}$. 
\end{rem}
\begin{ex}
{\rm Here we give some examples of sub-Riemannian geometries which fit in the above framework and where condition  
\eqref{CommutatorConstantCoefficents} holds:\\

%\begin{enumerate}
%\item \hspace{-1cm} 
$(1)$ 
The Heisenberg group: $X_1=(1,0,-\frac{y}{2})^T$ and
$X_2=(0,1, \frac{x}{2})^T$ on $(x,y,z)\in \mathbb{R}^3$. 
In this case  $Z_1=X_1$, $Z_2=X_2$ and $Z_3=Z:=[X_1,X_2]=(0,0,1)^T$. The family \(\{Z_1,Z_2,Z_3\}\) forms a basis 
for the Lie algebra (here \(n=N\)). %\\
One can calculate that $c_{kjl}=0$ for any $(k,j,l)\neq (1,2,3) \vee (2,1,3)$
while $c_{123}=1$ and $c_{213}=-1$.
%\item 
\\

\noindent$(2)$  The Gru\v{s}in plane:  $X_1=(1,0)^T$ and $X_2=(0,x)^T$ on
$(x,y)\in \mathbb{R}^2$. In this case $Z_1=X_1$, $Z_2=X_2$ 
and $Z_3=Z:=[X_1,X_2]=(0,1)^T$ and  $c_{kjl}=0$ for any $(k,j,l)\neq (1,2,3)
\vee (2,1,3)$ while $c_{123}=1$  and $c_{213}=-1$. 
The family \(\{Z_1,Z_2,Z_3\} \) contains a basis for the Lie algebra, given by
\(\{Z_1,Z_3\}\). \\

%\noindent$(3)$%\item 
%Trigonometric geometries: 
%$X_1=(1,0)^T$ and $X_2=(0,\cos x)^T$ on $(x,y)\in \mathbb{R}^2$. In this %case $Z_1=X_1$, $Z_2=X_2$ and 
%$Z_3=Z:=[X_1,X_2]=(0,-\sin x)^T$. Again $c_{kjl}=0$ for any $(k,j,l)\neq %(1,2,3) \vee (2,1,3)$ while $c_{123}=1$ and 
%$c_{213}=-1.$
%\\

%\item 
%\noindent$(4)$ 
%The rototraslation geometry:  $X_1=(\cos \theta, \sin \theta, 0)^T$ and
%$X_2=(0,0,1)^T$ 
%on $(x,y,\theta)\in \mathbb{R}^2\times S^1$. In this case $Z_1=X_1$, $Z_2=X_2$ %and 
%$Z_3=Z:=[X_1,X_2]=(\sin \theta, -\cos \theta, 0)^T$. One can check: $c_{kjl}=0$ 
%for any $(k,j,l)\neq (1,2,3) \vee (2,1,3)\vee (3,2,1)$ while $c_{123}=1$ %and 
%$c_{213}=c_{321}= -1$.\\

\noindent$(3)$ %\item 
The Martinet distribution: 
$X_1=(1,0,-y^2)^T$ and $X_2=(0,1, 0)^T$ on $(x,y,z)\in \mathbb{R}^3$. 
In this case $Z_1=X_1$, $Z_2=X_2$, $Z_3=Z:=[X_1,X_2]=(0,0,2y)^T$
and $Z_4=[Z,X_2]=(0,0,-2)^T$.
 Then $c_{kjl}=0$ for any $(k,j,l)\neq (1,2,3) \vee (2,1,3)\vee (3,2,4)$ while $c_{123}=c_{324}=1$ and $c_{213}=-1$.
%\end{enumerate}
\\

Note that the last example (Martinet distribution)  is  a step 3 distribution 
while all the others are step 2, and it is the easiest sub-Riemannian geometry where normal geodesics occur.}

\end{ex}
For smooth functions $f$, we define
\begin{equation}\label{Gammadef}
\Gamma( f):=\sum_{k=1}^n |Z_kf|^2,
\end{equation}
which  we call the complete gradient form (as opposed, for
example, 
to the sub-gradient of a Lie group).
Note here that in general it may be convenient later to include
more fields $Z_k$ than it would be necessary just to span the tangent space
at any given point.
The corresponding quadratic form is given by
$$\Gamma(f,g)=\sum_{k=1}^n \big(Z_kf\big)  \big(Z_kg\big).$$

\subsection{Associated Stochastic Differential Equation}\label{subs.1.1}
$\quad$\\

%added 07/11-> I changed some indiced e.g n by N 
Here we want to write the Stochastic Differential Equation having the operator
$\mathcal{L}$ as generator.
The SDE has the general form
$$
d\xi(t)=\mu(\xi(t))dt+A(\xi(t))\circ dW(t),
$$
where $\mu(\xi(t))\in \mathbb{R}^n$ is the so called drift part while $A(\xi(t))$
is a $N\times m$ matrix, $W$ is an $m$-dimensional Brownian motion and by $\circ$ we mean the Stratonovich differential.\\
It is known that, given a second-order differential operator, it is possible to find a SDE having such an operator as generator,
whenever the second-order part can be expressed as trace.
In general the first-order part of the operator is related to the drift-part (i.e. the deterministic part of the SDE) while 
the second-order part is related to the stochastic part of the equation. In particular, the stochastic part has to be written 
as a Stratonovich differential whenever there is an explicit dependence on the space. We  also recall that the Stratonovich 
differential can be always written in It\^o formulation as follows:
\begin{equation}
\label{Stratonovich2Ito}
A(\xi(t))\circ dW(t)= A(\xi(t))dW(t)+\sum_{i=1}^m\nabla_{A^i}A^i(\xi(t))\;dt,
\end{equation}
where $A^i$ are rows  of the matrix $A$ and $\nabla_{A^i}A^j$ is the derivative of the vector field $A^j$ along the vector 
field $A^i$, for any $i,j$.
Our operator can be written as
%added 07/11 -> changed L-I order and L-II order
$$
\mathcal{L}=\left(\sum_{i=1}^m X_i^2+\sum_{i,j=1}^m G_{i\,j}(x) X_iX_j
\right)+\left(\sum_{i=1}^m \alpha_i (x) X_i-\beta D\right)=:\mathcal{L}_{_{\textrm{\bf II-order}}}+\mathcal{L}_{_{\textrm{\bf I-order}}}.
%\textrm{II-order op.}+\;\textrm{I-order  op.}
$$
Note that to write the associated SDE, we do not need  any assumption on $D$   while we need to assume condition 
\eqref{Condition_onG}.\\

We denote by $\sigma(x)$ the $n\times m$ matrix whose rows are the vector fields $X_1(x),\dots,X_m(x)$. 
We first write the drift part which comes from the first-order part of the operator, that is, for  any smooth function $f$,
$$
\mathcal{L}_{_{\textrm{\bf I-order}}}\;f= \sum_{i=1}^m \alpha_i(x)X_i(f)-\beta D=
\alpha^T(x)\sigma^T(x)\nabla f-\beta D^T \nabla f.
$$
Note that $\sigma^T(x) \nabla f=:D_{\cX} f$ is the horizontal gradient of $f$ (or to be more precise is the coordinate-vector 
of the horizontal gradient $\cX f$ written in the basis of the vector fields $X_1,\dots,X_m$).
The drift part for the associated SDE is:
$$
\mu(\xi(t))=\alpha(\xi(t))\sigma(\xi(t))-\beta D(\xi(t)).
$$
\par Now we want to write explicitly the stochastic part of the equation. Let us first assume that
$G$ is symmetric (i.e. $G=G^*$ )
and  introduce
$$
B(x):=\sqrt{I+G(x)},
$$
where $I$ is the $m\times m$-identity-matrix.
Note that $B=(B_{ij})_{i,j=1}^m$ is well-defined since $I+G$ is symmetric and we have assumed condition  \eqref{Condition_onG}.
We are going to show that  
$$\mathcal{L}_{_{\textrm{\bf II-order}}}:= \sum_{i=1}^m X_i^2+\sum_{i,j=1}^m G_{i\,j}(x) X_iX_j
$$
is the generator of 
$$
d\xi(t)=\sigma(\xi(t))B(\xi(t))\;\circ d W(t)
=\left(\sum_{i,j=1}^m B_{j\,i}(\xi(t))X^l_j(\xi(t))\circ dW_i\right)_{l=1}^N\!\!\!\!\!\!\!= \sum_{i=1}^mY_i\circ dW_i,$$
with $W_i$ the standard Brownian motion and  denoting
 $Y_i:=\sum_{j=1}^m B_{j\,i}X_j$.
It is known (see e.g. \cite{F-W}, \cite{F-S}) that the generator of the stochastic equation  
$d\xi_i=Y_i\circ dW_i$
is given by $\sum_{i} Y_i^2$, so we rest just to calculate it:
\begin{align*}
\sum_{i} Y_i^2=\sum_{i} \left(\sum_jB_{ij} X_j\right)^2=\sum_i \sum_{lm} B_{il}X_lB_{im} X_m
&=
\sum_{lm} \sum_i B_{il}B_{il}X_l X_m\\
&=
\sum_{lm} (B\,B^t)_{lm}X_l X_m.
\end{align*}
Using the fact that $G$ is symmetric, we get
\begin{align*}
\sum_{i} Y_i^2&=\sum_{lm} (B^2)_{lm}X_l X_m=
\sum_{lm} (I+G)_{lm}X_l X_m=
\sum_lX_l^2+\sum_{lm}G_{lm} X_lX_m\\&= \mathcal{L}_{_{\textrm{\bf II-order}}}.
\end{align*}
Therefore, under the assumption that $G$ is symmetric (and so is $B$), the associated SDE is 
\begin{equation}
d\xi(t)=\big(\alpha(\xi(t))\sigma(\xi(t))-\beta D(\xi(t))\big) dt+\sigma(\xi(t))\big(I+G(\xi(t))\big)\circ d W(t).
\end{equation}
Let us now see what happens when the matrix $G$ is not symmetric.
Note that 
 \begin{multline*}
\sum_{i,j} G_{ij}X_iX_j=
\sum_{i,j}\left(\frac{G_{ij}+G_{ji}}{2}\right)X_iX_j+ \sum_{i,j}\left(\frac{G_{ij}-G_{ji}}{2}\right)X_iX_j\\
\equiv \sum_{i,j}G^*_{ij} X_iX_j+ \sum_{i,j}G^{aSym}_{ij} X_iX_j.
\end{multline*}
Since for $G^{aSym}$, the antisymmetric part of $G$,  we have
$$
\sum_{i,j}G^{aSym}_{ij} X_iX_j = \frac12 \sum_{i,j}G^{aSym}_{ij} [X_i,X_j],
$$
therefore the antisymmetric part of $G$ gives an extra first order part
(i.e. an extra term in the  drift part) depending on the  commutators.
Thus, under assumption \eqref{Condition_onG}, the SDE associated to the operator $\mathcal{L}$ is
\begin{multline}
\label{SDE}
\!\!\!\!d\xi(t)\!=\!\!\left\{\!\!\alpha(\xi(t))\sigma(\xi(t))-\beta D(\xi(t)) 
\!\!+\!\!\!
\frac12\sum_{i,j=1}^mG^{aSym}_{ij}\big[X_i(\xi(t)),X_j(\xi(t))\big]\!\!\right\} \!dt\\
+\sigma(\xi(t))\big(I+G^*(\xi(t))\big)\;\circ \;d W(t).
\end{multline}
\begin{rem}\label{rem1.3}
Without assumption \eqref{Condition_onG} the II-order
part of the operator cannot be written as a trace and therefore is not a generator of a stochastic process. 
The same condition will arise in order to find an exponential decay for the semigroup associated to the operator.
\end{rem}
%%%
\subsection{Existence of a limit measure}\label{limitmeasurefinitedim}
\label{subs.1.2}
Let \((P_t)_{t\ge 0} \) denote the semigroup generated by \(\mathcal{L}\), where
\(\mathcal{L}\) is given by \eqref{Operator}. 
We show that one can extract a subsequence \((P_{t_{k}} )_{k=1}^\infty\) which converges weakly to a probability measure 
on \(\mathbb{R}^N\).
%added
Here and in the sequel we use the notation \(d(x)=d(x,0),\) where \(d\) is a metric on \(\mathbb{R}^N\). 
\begin{lemma} \label{Ptdbound} \label{lem1.4}
Let \(\rho\) be a smooth function such that \(\rho(x)=0\) for \(d(x)<1\) and 
\(\rho \rightarrow \infty\) as \(d(x) \rightarrow \infty\). Assume  
\begin{enumerate}
\item $\sum_{i=1}^mX_i^2 \rho +\sum_{i,j=1}^m G_{ij}X_iX_j\rho\le C_{1} $
\item $\sum _{i=1}^m \vert X_i \rho \vert^2 \le C_2$
\item $ P_t \rho \le c_1P_tD\rho+c_2  $
\end{enumerate}
for some constants \(C_{1},C_{2},c_1,c_2>0\).  Then there exists a constant 
\(K\in(0,\infty)\) such that $ P_t \rho \le K$ for all \(t>0.\) 
\end{lemma} 

\proof We have
\begin{align*}
\partial_t P_t\rho&=P_t \mathcal{L}\rho\\&= P_t \sum_{i=1}^m X_i^2 \rho-\beta P_t D\rho+P_t\sum_{i=1}^m \alpha _iX_i\rho +
P_t\sum_{i,j=1}^m G_{ij}X_iX_{j} \rho\\ 
& \le C_{1}-  \frac{\beta}{c_{1}}P_t\rho-\frac{c_2\beta}{c_1}
+\max_i\Vert \alpha_i \Vert_\infty \sqrt{n}  C_2
\end{align*}
using our assumptions. Integrating this inequality we get 
\[P_t\rho\le  e^{-\eta t}\rho+\frac{\beta}{c_1\eta}\left(1-e^{-\eta t}\right) \]
with \(\eta=C_1-\frac{c_2\beta}{c_1}+\max_i\Vert \alpha_i \Vert_\infty \sqrt{n}  C_2\),
which is bounded for all \(t>0.\)
\endproof
\begin{rem} \label{rem1.5}
The first two assumptions of Lemma \ref{Ptdbound} can be relaxed to 
\begin{enumerate}
\item $\sum_{i=1}^mX_i^2 \rho +\sum_{i,j=1}^m G_{ij}X_iX_j\rho\le C_{1}\rho+ \tilde{C}_1$ and
\item $\sum _{i=1}^m \vert X_i \rho \vert^2 \le C_2\rho +\tilde{C}_2$
\end{enumerate} 
respectively, for some constants \(C_1,\tilde{C}_1,C_2, \tilde{C}_2>0\), at the expense of having to take \(\beta\) 
large enough to ensure that the coefficient of \(P_t \rho\) in the proof is negative.
\end{rem}
The function \(\rho\) can be thought of as a cut-off of an appropriate distance
function. 
\begin{ex} 
We illustrate this in case of a Lie group of Heisenberg type
%added 07/11 -> changed n by l and N(x,t) by G(x,t) since N was confusing in the following formulas (used for  the dimension of the space). G is usually used for gauge norm which is the oen we use.
 \(\mathbb{G}= \left( \mathbb{R}^{m+l},\circ, \delta_\lambda\right)\), 
with  left-invariant vector fields \(X_1,\dots,X_m\). 
Such a group is naturally equipped with dilations \(\delta_{\lambda}(x,t)=(\lambda x, \lambda^2 t)\), 
where \((x,t ) \in \mathbb{R}^m \times \mathbb{R}^l\), which form a 1-parameter family of homomorphisms. 
Here, one may define the following smooth homogeneous gauge\ (also known as the Folland-Kaplan gauge, see e.g. \cite{B-L-U})
\begin{equation}
N(x,t)=\left(\vert x \vert ^4+16 \vert t\vert^2\right)^\frac{1}{4},
\end{equation}
where \(\vert \cdot \vert \) denotes the Euclidean norm.
A computation then shows that the sub-gradient and the sub-Laplacian of this gauge function read
\begin{equation}
 \sum_{i=1}^m \vert X_i N \vert ^2= \frac{\vert x\vert ^2}{N^2}
\end{equation}
 and
\begin{equation}  
\sum_{i=1}^{m} X_i^2N=3\frac{\vert x \vert ^2}{N^3}
\end{equation}
respectively, while the dilation operator is the generator of \((\delta_\lambda)_{\lambda >0}\) given by 
\begin{align*}
D &=\partial_\lambda  \mid_{_{\lambda=1}}\delta_\lambda(x,t)\\
&=\sum_{i=1}^m x_i \partial_{x_i}+2 \sum_{i=1}^n t_i \partial _{t_i}.
\end{align*}
Since \(\partial_{x_i}N=N^{-3}\vert x \vert ^2x_i\) and \(\partial_{t_i}N=8N^{-3}t_i\), we have 
\begin{align*}
DN&= N^{-3}\left( \vert x \vert ^4+16 \vert t \vert ^2\right)=N
\end{align*}
and therefore \(P_t N=P_t DN\). 
Moreover, if we introduce a cut-off function \(g: \mathbb{R} \rightarrow \mathbb{R}\) such that  
\(g(x)=0\) on \([0,1]\), \(g(x)=x\) for \(x \ge 2\) and \(g \) is continuous and
smooth on \((1,2)\), then the function \(\hat{N}(x,t)=g(N(x,t))\) 
is such that \(\sum_{i=1}^mX_i^2 \hat{N}+ \vert X_i \hat{N} \vert^2\) is
bounded, since \(\vert x \vert \le N(x,t)\). 
Hence, in this case,  a function satisfying the assumptions of Lemma \ref{Ptdbound}  exists. 
\end{ex}
Similarly one can construct suitable $\rho$ for other (noncompact) homogeneous
Lie groups using a smooth (outside the origin)
homogeneous norm (of \cite{H-S}, \cite{B-L-U}).

\begin{teo}\label{thm1.5} \ 
There exists a sequence \(\{t_k\}_{k=1}^ \infty \subset \mathbb{R}\) and a
probability measure 
\(\nu\) on \(\mathbb{R}^n\) such that for all bounded and Lipschitz \(f\) 
\[
P_{t_k} f \rightarrow \int f d\nu\ 
\] 
as \(k \rightarrow \infty.\)
\end{teo}

\proof For \(L>0, \) we define sets \(\Upsilon_L=\{\rho \le L\}\) for which we have,
by Markov's inequality and Lemma \ref{Ptdbound}, 
\[P_t (\Upsilon_L) \ge 1-\frac{K}{L},\] for some constant \(K>0\). Therefore
\((P_t)_{t>0}\) represents a tight family of measures on $\mathbb{R}^N$ 
and we deduce from Prokhorov's Theorem that there exists a convergent subsequence 
$P_{t_k} \rightarrow \lim_{k \rightarrow \infty}P_{t_{k}} =: \nu$ in the weak
sense. 
\endproof

\subsection{Complete Gradient Bounds %and Exponential decay for the associated semigroup
}\label{subs1.3}

We start by proving the following bound for the semigroup \(P_t\) and
the complete gradient 
\(\Gamma\) defined in \eqref{Gammadef}. 
  
\begin{teo} \label{thm1.6}
\label{Result_Constant Coefficients} 
Let $\mathcal{L}$ be the operator defined in \eqref{Operator}, under the
assumptions  \eqref{Condition_onD} and \eqref{Condition_onG} 
and let $P_t$ be the semigroup associated to $\mathcal{L}$. Let us also assume that  $G_{ij}$ is constant and 
$\Vert Z_k \alpha_i \Vert _\infty < \infty$ for all \(k=1\dots n\) and \(i,j=1 \dots m\).
If \eqref{CommutatorConstantCoefficents} holds, then 
there exists $\kappa\in \mathbb{R}$ such that 
\begin{equation}
\label{Stima1}
\Gamma(P_tf)\leq e^{-\kappa t} P_t \Gamma(f).
\end{equation}
Moreover, there exists \(b_0\in(0,\infty)\) such that for all \(\beta>b_0\)  we
have \(\kappa\in(0,\infty).\)  
\end{teo}

\begin{proof}
The proof follows the Bakry-Emery type strategy (see e.g.
\cite{A-B-C-F-G-M-R-S}, \cite{B-E})
with suitable modifications required by our setup. Let us set $f_s:=P_sf$. 
Note that it is sufficient to prove that

\begin{equation}
\label{Stima1_formaDifferenziale}
\frac{d}{ds} P_{t-s}\Gamma(f_s)\leq -\kappa\ P_{t-s}\Gamma(f_s),
\end{equation}
which  gives  estimate \eqref{Stima1} after integration over $s\in [0,t]$.\par
To prove \eqref{Stima1_formaDifferenziale}, we remark that
$$
\frac{d}{ds} P_{t-s}\Gamma(f_s)=P_{t-s}
\big(-\mathcal{L}\Gamma(f_s)+2\Gamma(f_s,\mathcal{L}f_s)
\big),
$$
since $\Gamma(f,g)$ is defined as a bilinear form.\\
Using the explicit expressions for $\Gamma(f)$ and $\mathcal{L}$, the previous
relation becomes:
\begin{multline*}
 \frac{d}{ds} P_{t-s}\Gamma(f_s)= P_{t-s} \sum_k\bigg(-\mathcal{L}|Z_kf_s|^2+ 2(Z_kf_s)(Z_k\mathcal{L}f_s)\bigg)\\
 = P_{t-s} \sum_k\bigg(-\mathcal{L}|Z_kf_s|^2 
+ 2(Z_kf_s)(\mathcal{L}Z_kf_s)+2(Z_kf_s)\big([Z_k,\mathcal{L}]f_s\big)\bigg).
\end{multline*}
We note that
\begin{multline}
\label{I1}
\!\!\!\!\!\!I_k:=\!\! -\mathcal{L}|Z_kf_s|^2+ 2(Z_kf_s)(\mathcal{L}Z_kf_s)\!\!=\!-2\sum_j |X_j Z_kf_s|^2
-2\!\sum_{i,j} G_{ij} (X_iZ_kf_s)(X_jZ_kf_s)\\= -2\sum_{i,j} \big(G_{ij}+\delta_{ij}\big)( X_iZ_kf_s)(X_jZ_kf_s),
\end{multline}
where $\delta_{ij}=0$, for $i\neq j$ and $\delta_{ii}=1$, for any $i,j=1,\dots,m$.\\
The third term is more difficult to estimate since it depends on the commutators. For this purpose, we need  
to  use assumption \eqref{CommutatorConstantCoefficents}.
Let us set 
$$J_k:=2(Z_kf_s)\big([Z_k,\mathcal{L}]f_s\big),$$ 
and  calculate the commutators $[Z_k,\mathcal{L}]f_s$, i.e.
$$
[Z_k,\mathcal{L}]f_s=[Z_k,L]f_s+[Z_k,L_G]f_s+[Z_k,L_{\alpha}]f_s.
$$
Recalling the definition of $\mathcal{L}$ and noting that
$[Z,Y^2]f=[Z,Y]Yf+Y[Z,Y]f$, using assumption \eqref{Condition_onD}, 
for the operator $L$ we get
\begin{align*}
[Z_k,L]f_s&=\sum_j [Z_k,X_j^2]f_s-\beta\lambda_k (Z_kf_s)\\ &=
\sum_j \big\{[Z_k,X_j](X_jf_s)+ X_j[Z_k,X_j]f_s\big\}-\beta\lambda_k (Z_kf_s).
\end{align*}
Using  condition \eqref{CommutatorConstantCoefficents}, we obtain
$$
[Z_k,L]f_s=\sum_{j,l}c_{kjl} \big\{(Z_lX_jf_s)+(X_jZ_l f_s)
\big\}
-\beta\lambda_k (Z_kf_s).
$$
We are going to use the negative term $I_k$, in order  to control mixed terms like
$(Z_kf_s)(X_jZ_kf_s)$. To this end we have to rewrite the term $(Z_lX_jf_s)$ in
a more suitable form, using once more assumption
 \eqref{CommutatorConstantCoefficents}, i.e.
\begin{align*}
\sum_{j,l}c_{kjl} \big\{(Z_lX_jf_s)+(X_jZ_l f_s)
\big\}
&=\sum_{j,l}c_{kjl} \big\{2(X_jZ_l f_s)+[Z_l,X_j]f_s
\big\}\\
&=
2\sum_{j,l}c_{kjl} (X_jZ_l f_s)+\sum_{j,l,n}c_{kjl} c_{ljn}
(Z_nf_s),
\end{align*}
Hence, summing up we get
\begin{align}
\label{J_1}
\sum_kJ_k =&-2\beta\sum_k\lambda_k |Z_kf_s|^2+
4\sum_{k,j,l}c_{kjl} (Z_kf_s)(X_jZ_l f_s)\\&+2\sum_{k,j,l,n}c_{kjl} c_{ljn}
(Z_kf_s)(Z_nf_s)
+2\sum_k(Z_kf_s)\big\{[Z_k,L_G]f_s+[Z_k,L_{\alpha}]f_s\big\}.
\end{align}
We can now similarly  estimate the remaining terms.
Note that 
$$[Z,XY]f=[Z,X]Yf+X[Z,Y]f \, ,$$
and thus
\begin{align*}
[Z_k,L_G]f_s&=\sum_{i,j} G_{ij}\big\{[Z_k,X_i](X_jf_s)+X_i[Z_k,X_j]f_s\big\}\\
&=
\sum_{i,j,l} (G_{ij}+G_{ji})c_{kil}(X_jZ_lf_s)+\sum_{i,j,l,n} G_{ij}c_{kil}c_{ljn}(Z_nf_s).
\end{align*}
Moreover, 
\begin{align*}
[Z_k,L_{\alpha}]f_s&=\sum_{i}(Z_k\alpha_i)X_if_s+\sum_i\alpha_i[Z_k,X_i]f_s \\ 
&=\sum_i(Z_k\alpha_i)X_if_s+\sum_{i,l}\alpha_ic_{kil}(Z_lf_s).
\end{align*}
Therefore, 
\begin{align*}
2\sum_k(Z_k f_s)[Z_k,L_\alpha]f_s &= 2\sum_{k,i} (Z_k f_s)(Z_k \alpha_i)(X_i
f_s) + 2\sum_{k,i,l}\alpha_ic_{kil} (Z_kf_s) (Z_lf_s) \\ 
& \le  \sum_{k,i} \Vert Z_k\alpha_i \Vert_\infty \left( \vert Z_kf_s \vert^2
+\vert X_i f_s \vert ^2 \right)
+ 2 \sum_{k,i,l}\alpha_ic_{kil} (Z_kf_s) (Z_lf_s)
\\ & \le \max_k \sum_i \Vert Z_k \alpha_i \Vert _\infty \Gamma(f_s)+\max_i \sum_k \Vert Z_k \alpha_i \Vert_\infty \Gamma(f_s)
\\&\quad+2 \sum_{k,i,l}\alpha_ic_{kil} (Z_kf_s) (Z_lf_s) ,
\end{align*}
where we used Young's inequality to estimate the first term.
Combining the above, \eqref{J_1} becomes
\begin{align}
\label{J_2} \nonumber
\sum_kJ_k=&-2\beta\sum_k\lambda_k |Z_kf_s|^2
+4\sum_{k,j,l}c_{kjl} (Z_kf_s)(X_jZ_l f_s)+\\ \nonumber &
+2\sum_{k,j,l,n}c_{kjl} c_{ljn}
(Z_kf_s)(Z_nf_s) +2 \sum_{k,i,j,l}\!(G_{ij}+G_{ji})c_{kil}(Z_kf_s)(X_jZ_lf_s) 
\\&
+2\sum_{k,i,j,l,n} G_{ij}c_{kil}c_{ljn}(Z_kf_s)(Z_nf_s)+2\sum_{k,i,l}\alpha_ic_{kil}(Z_kf_s)(Z_lf_s)+  \eta\Gamma(f_s), 
\end{align}
with \(\eta=\max_k \sum_i \Vert Z_k \alpha_i \Vert _\infty +\max_i \sum_k \Vert Z_k \alpha_i \Vert_\infty \).
We start by estimating  the terms in \eqref{J_2} where $(X_jZ_lf_s)$ does not appear.
Let us set $\lambda_* :=\min_k\lambda_k>0$, then 
$-2\beta\sum_k\lambda_k |Z_kf_s|^2
\leq -2\beta \lambda_* \Gamma(f_s)$ (recalling that $\beta>0$).
The other terms can be treated similarly. Recalling that  $\alpha_i$, $G_{ij}$
and $c_{kjl}$ are in general 
non positive, we get:
\begin{align*}
2\sum_{k,j,l,n}&c_{kjl} c_{ljn}
(Z_kf_s)(Z_nf_s)+2\sum_{k,i,j,l,n} G_{ij}c_{kil}c_{ljn}(Z_kf_s)(Z_nf_s)\\
&=2\sum_{k,i,j,l,n} \big(G_{ij}+\delta_{ij}\big)c_{kil}c_{ljn}(Z_kf_s)(Z_nf_s)
\\&\leq \sum_{k,i,j,l,n}\!\!\!\! \big|G_{ij}+\delta_{ij}\big||c_{kil}||c_{ljn}|\big\{|Z_nf_s|^2+|Z_kf_s|^2\big\}\\
&\leq\sup_k\sum_{n,i,j,l}\big|G_{ij}+\delta_{ij}\big|\big\{|c_{kil}||c_{ljn}|+|c_{nil}||c_{ljk}|\big\}\Gamma(f_s)=:C_1\Gamma(f_s)
\end{align*}
and
$$
2\sum_{k,i,l}\alpha_ic_{kil}(Z_kf_s)(Z_lf_s)\leq \sup_k \sum_{i,l} |\alpha_i|\big(|c_{kil}|+|c_{lik}|\big)\Gamma(f_s)=:C_2 \Gamma(f_s).
$$
By \eqref{I1} and \eqref{J_2}, 
we have  the following estimate:
\begin{align*}
 \frac{d}{ds} P_{t-s}\Gamma(f_s)&=\sum_k(I_k+J_k)
\\& \leq
 \big(-2\beta \lambda_* + C_1+C_2+\eta\big)\Gamma(f_s)
 +4\sum_{k,j,l}c_{kjl} (Z_kf_s)(X_jZ_l f_s)
\\&\quad +2 \sum_{k,i,j,l}\!(G_{ij}+G_{ji})c_{kil}(Z_kf_s)(X_jZ_lf_s)
 \\& \quad-2 \sum_{k,i,j} \big(G_{ij}+\delta_{ij}\big) (X_iZ_kf_s)(X_jZ_kf_s).
\end{align*}
The idea is to use Young's inequality to estimate the remaining terms with a
non-positive part depending on $ (X_iZ_kf_s)$ 
and a positive part depending just on $\Gamma(f_s)$. 
Let us recall that $D_{\cX} Z_kf_s=(X_1Z_kf_s,\dots,X_1Z_kf_s)\in \mathbb{R}^m$
is the horizontal gradient of $Z_k f_s$, then
$$
\sum_kI_k=\! -2\! \sum_{k,i,j} \big(G_{ij}+\delta_{ij}\big) (X_iZ_kf_s)(X_jZ_kf_s)=\!-2\!\sum_k\bigg< \big(G^*+I) D_{\cX}Z_kf_s, D_{\cX} Z_kf_s\bigg>,
$$
since $\big<Aa,a\big>=0$, for any $a\in \mathbb{R}^m$, whenever $A$ is an
antisymmetric matrix (recall 
%that $G^*$ is  the symmetric part of $G$, i.e. 
$G^*_{ij}=\frac{G_{ij}+G_{ji}}{2}$).\\
Analogously,
\begin{multline*}
4\sum_{k,j,l}c_{kjl} (Z_kf_s)(X_jZ_l f_s)
 +2 \sum_{k,i,j,l}(G_{ij}+G_{ji})\;c_{kil}(Z_kf_s)(X_jZ_lf_s)\\
 =4
 \sum_{k,i,j,l}(\delta_{ij}+G^*_{ij})\;c_{kil}(Z_kf_s)(X_jZ_lf_s)=:I'.
\end{multline*}
For the sake of simplicity, let us denote $a_{jl}:=\sum_{k,i} (\delta_{ij}+G^*_{ij})c_{kil}(Z_kf_s)$ and 
$b_{jl}:=(X_jZ_lf_s)$.  Young's inequality tells that 
$a_{jl}b_{jl}\leq \frac{\varepsilon a_{jl}^2}{2}+\frac{b_{jl}^2}{2\varepsilon }$, for any $\varepsilon >0$. 
Thus 
$$
I'\leq 2\sum_{jl}\left(
\varepsilon \left(\sum_{k,i} (\delta_{ij}+G^*_{ij})\;c_{kil}(Z_kf_s)\right)^2+\frac{|X_jZ_lf_s|^2}{\varepsilon}
\right).
$$
We can estimate  the first part  as follows:
\begin{multline*} 2\varepsilon \sum_{jl}
2\left(\sum_{k,i} (\delta_{ij}+G^*_{ij})\;c_{kil}(Z_kf_s)\right)^2\\
=2\varepsilon\sum_{j,l}\left(
\sum_{k,n}\left[ \sum_i \big(\delta_{ij}+G^*_{ij}\big)c_{kil}(Z_nf_s)\right]\left[
\sum_i \big(\delta_{ij}+G^*_{ij}\big)c_{nil}(Z_kf_s)\right]
\right)\\
\leq \varepsilon\sum_{j,l,k,n}\left\{
\left[ \sum_i \big(\delta_{ij}+G^*_{ij}\big)c_{kil}\right]^2|Z_nf_s|^2+
\left[ \sum_i \big(\delta_{ij}+G^*_{ij}\big)c_{nil}\right]^2|Z_kf_s|^2
\right\}
\\
=2\varepsilon\sum_{k,l,j}\left[ \sum_i \big(\delta_{ij}+G^*_{ij}\big)c_{kil}\right]^2\Gamma(f_s)
=:\varepsilon\; C_3 \Gamma(f_s),
\end{multline*}
while
$$
2\sum_{jl}
\frac{|X_jZ_lf_s|^2}{\varepsilon}=
2 \sum_k \bigg<\frac{1}{\varepsilon}
D_{\cX}Z_kf_s, D_{\cX} Z_kf_s\bigg>.
$$
Therefore we can conclude
\begin{multline*} 
\frac{d}{ds} P_{t-s}\Gamma(f_s)=\sum_k(I_k+J_k)\\
 \leq
 \big(-2\beta \lambda_* + C_1+C_2+\varepsilon\; C_3\big)\Gamma(f_s)
-2\sum_k\bigg< \left(G^*+I-\frac{I}{\varepsilon}\right) D_{\cX}Z_kf_s, D_{\cX} Z_kf_s\bigg>.
\end{multline*}
By assumption \eqref{Condition_onG}, there exists $\delta>0$ such that 
\begin{equation}
\label{Condition_onGBis}
G^*+I\geq \delta I.
\end{equation}
Choosing $\varepsilon=\frac{1}{\delta}$, where $\delta$ is (the biggest number) such that  \eqref{Condition_onGBis} holds, 
the inner product in the above estimate  is non-negative.  Therefore
$$
\frac{d}{ds} P_{t-s}\Gamma(f_s)
 \leq
 \left(-2\beta \lambda_* + C_1+C_2+\eta+\frac{1}{\delta}\;
C_3\right)\Gamma(f_s),
$$
which gives the theorem with
\begin{align}\label{mconstant}
\nonumber \kappa:=&2\beta \lambda_* - C_1-C_2-\eta-\frac{1}{\delta}\; C_3\\
\nonumber=
&2\beta\min_k\lambda_k-\sup_k\sum_{n,i,j,l}\big|G_{ij}+\delta_{ij}\big|\big\{|c_{kil}||c_{ljn}|+|c_{nil}||c_{ljk}|\big\}\\
\nonumber&- \sup_k \sum_{i,l} |\alpha_i|\big(|c_{kil}|+|c_{lik}|\big)
-\max_k \sum_i \Vert Z_k \alpha_i \Vert _\infty 
-\max_i \sum_k \Vert Z_k \alpha_i \Vert_\infty \\&-\frac{2}{\delta}\;
\sum_{k,l,j}\left[ \sum_i \big(\delta_{ij}+G^*_{ij}\big)c_{kil}\right]^2.
\end{align}
Finally, by choosing 
$\beta>\frac{1}{2\lambda_*}\left(C_1+C_2+\eta+\frac{1}{\delta}\; C_3\right)$ we
can ensure that 
\(\kappa>0\). 
\end{proof}
%added
\begin{rem}\label{nonconstantG}
More generally, for a non-constant matrix \(G=G(x)\)  the theorem  continues to hold under the additional assumption that the quantities \(\Vert Z_kG_{ij} \Vert_\infty \) are bounded for all \(k=1\dots n\) and \(i,j=1\dots m\).\end{rem}
\begin{rem} [Case $G=0$] \label{rem1.7}
Whenever $G=0$, we can choose $\delta=1$ in the constant $\kappa$. Then 
\begin{align*}
\kappa=&2\beta\min_k\lambda_k- \max_k \sum_i \Vert Z_k \alpha_i \Vert _\infty -\max_i \sum_k \Vert Z_k \alpha_i \Vert_\infty \\
&-\sup_k\left\{\sum_{n,j,l}\big\{|c_{kjl}||c_{ljn}|+|c_{njl}||c_{ljk}|\big\}
- \sum_{i,l} |\alpha_i|\big(|c_{kil}|+|c_{lik}|\big)\right\}
-2\sum_{k,l,j}c_{kjl}^2.
\end{align*}
\end{rem}
\begin{rem} [Optimal constant in the case $G=0$] \label{rem1.8}
The constant found is a priori  not always optimal. In fact, to deduce constant
$C_1$ and $C_3$ we used two different estimates:
$2\mu\nu ab\leq \mu\nu (a^2+b^2)$ for the first constant, $2\mu\nu ab\leq
\mu^2a^2+\nu^2 b^2$ for the second one.\\
It is possible to give examples when the first estimate is optimal (i.e. if
$a=b$ but $\mu\neq\nu$ e.g. $\mu= \nu^{-1}$) 
and examples when the second one is the optimal one (i.e. $\mu a=\nu b$ but 
$a\neq b$ e.g. $a=b^{-1}$).
Therefore we could use both  estimates, finding two different constants:
\begin{align*}
&C_1'=2\sum_{kjl}c_{kjl}^2=C_2,\\
&C_2'=2\sup_{n}\sum_{kjl}c_{kjl}c_{njl}.
\end{align*}
It is clear that $C_2<C_2'$ (taking $k=n$ in $C'_2$)  while $C_1$ is often smaller than $C_1'$ 
because in general very few $c_{kjl}$ are  different from 0.  Therefore in
$C_1$ many terms vanish, but it is not a always true.
Therefore in the case $G_{ij}=0$, the optimal constant is given by
$\overline{\kappa}:=2\beta \lambda_*-\overline{C}_1-C_2- C_3-\eta$, with
$\overline{C}_1:=\min\{C_1,C'_1\}$.
In the same way, one could write the optimal constant in the case $G_{ij}\neq 0$.
\end{rem}

\begin{ex} \label{ex2.10}
$(1)$ In the Heisenberg group, one can consider the operator
$D:=xX_1+yX_2+2zZ=x\partial_x+y\partial_y+ 2z\partial_z$ which satisfies 
assumption  \eqref{Condition_onD} with $\lambda_1=1$, $\lambda_2=1$ and
$\lambda_3=2$; therefore $\lambda_*:=\min_k\lambda_k=1$. 
% added
In the simplest case where \(\alpha_i \equiv G_{ij} \equiv 0\), using Remark \ref{rem1.7} and recalling that \(c_{123}=1=-c_{213}\) and \(c_{kjl}=0\) otherwise, we see that \(\kappa=2\beta-4\) so that \(\kappa>0\) for any \(\beta>2\). \\
%added
$(2)$ In the Gru\v sin plane the dilation operator is given by $D:=x\partial_x +2y\partial_y. $  Assumption  \eqref{Condition_onD} is satisfied with  \(\lambda_1=\lambda_2=1\) and \(\lambda_3=2\) and thus \(\lambda_*=1\) as in the Heisenberg group.\\
$(3)$ The dilation operator for the Martinet distribution is \(D:= x\partial _x+y \partial_y+3z\partial_z   \) so \(\lambda_1=\lambda_2=1,\)  while \(\lambda_3=2\) and \(\lambda_4=3\). Hence \(\lambda_*=1.\)
\end{ex}
If we do not assume 
\eqref{CommutatorConstantCoefficents}, then 
by the H\"ormander condition we know that 
\begin{equation}
\label{Hoermander}
\exists \;\; c_{kjl}\in C^{\infty}(\mathbb{R}^n) \;\; \textrm{such that}\; \;[Z_k,X_j]=\sum_{l=1}^n c_{kjl}(x) Z_l.
\end{equation}
 Looking at the above calculations, we get and extra term in $\sum_kJ_k$ where the horizontal derivatives of the coefficients 
appear, more precisely
 $$
 2\sum_{k,j,l}(X_jc_{kjl})(Z_lf_s)(Z_kf_s)\leq \sup_k \sum_{jl}\big\{|X_j c_{kjl}(x)|+|X_j c_{ljk}(x)|\}\Gamma(f_s)=:C_4(x)\Gamma (f_s).
 $$
 This implies that:
 $$
\frac{d}{ds} P_{t-s}\Gamma(f_s)
 \leq
 \left(-2\beta \lambda_*+C_1(x)+C_2(x)+\eta+\frac{1}{\delta}\;
C_3(x)+C_4(x)\right)\Gamma(f_s),
$$
Therefore Theorem \ref{Result_NonConstantCoefficients} holds in a stronger form
since the exponential in the estimate depends now on time and space.
\begin{teo} 
\label{Result_NonConstantCoefficients} \label{thm2.11}
Let $X_1,\dots, X_m$ be smooth vector fields satisfying the H\"ormander condition, $\mathcal{L}$ the operator defined in 
\eqref{Operator}, satisfying  the assumptions
\eqref{Condition_onD}  and \eqref{Condition_onG}, 
and let $P_t$ be the associated semigroup. Let us also assume
that  $G_{ij}$ is constant and 
$\Vert Z_k \alpha_i \Vert _\infty < \infty$ for all \(k=1\dots n\) and \(i,j=1 \dots m\).
Then there exists a smooth function $\kappa(x)$ such that 
\begin{equation}
\label{Stima1_x}
\Gamma(P_tf)\leq e^{-\kappa (x)t} P_t \Gamma(f).
\end{equation}
Moreover, under the additional assumption that the functions $c_{kjl}(x)$ and
their horizontal derivatives $X_j c_{kjl}(x)$ 
are bounded in $x\in \mathbb{R}^n$, there exists \(\widehat{b_0}\in(0,\infty)\)
such that for all \(\beta > \widehat{b_0}\) 
we have $\kappa>0$.
\end{teo}
\par The following result provides an extension of  Theorem \ref{Result_Constant Coefficients}
to  an $l_q$-gradient bound. 
%Here, the length of the gradient is interpreted as 
%$\vert \nabla f \vert =\sqrt{\Gamma(f)}$. 
%
\begin{teo}
\label{L_q_Result_Constant Coefficients} \label{thm2.12}
Let $q >1.$ Under the assumptions of Theorem \ref{Result_Constant Coefficients}, there
exists a constant $\kappa' \in \mathbb{R}$ such that
\begin{equation*}\Gamma (P_t f)^{\frac{q}{2}}  \le
e^{-\kappa 't}P_t \Gamma (f)^{\frac{q}{2}}  .
\end{equation*}
Moreover, there exists $b'_0>0$ such that   for $\beta>b'_0 $ we have $\kappa'>0$.
\end{teo}
\begin{proof} 
As before, we follow the strategy outlined in \cite{A-B-C-F-G-M-R-S}. For simplicity we treat the case 
$L_G=L_\alpha=0$. The proof
of the general case follows from the proof of Theorem
\ref{Result_Constant Coefficients} and  arguments similar to the  ones below. We aim to show that 
\begin{equation*}
\frac{d}{ds}P_{t-s} \Gamma ( f_{s} )^{\frac{q}{2}}  \le -\kappa ' P_{t-s} \Gamma ( f )^{\frac{q}{2}}  
\end{equation*}
with $f_s=P_sf$ as above. To this end, we  note that  
\begin{align*}     \frac{d}{ds} P_{t-s} \Gamma(f_s)^{\frac{q}{2}}  
&= P_{t-s} \left(-\mathcal{L} \Gamma(f_s)^{\frac{q}{2}}
+q\Gamma(f_s)^{\frac{q}{2}-1}\Gamma(f_s,  \mathcal{L}f_s)\right) \\
&= P_{t-s} \left( 
-\frac{q}{2}\frac{\mathcal{L}\Gamma(f_s)}{\Gamma(f_s)^{1-\frac{q}{2}}}
+\frac{q}{2}\left(1-\frac{q}{2}\right)\frac{\bar{\Gamma}(\Gamma(f_s))}{
\Gamma(f_s)^{2-\frac{q}{2}}}
+ q\frac{\Gamma(f_s,  \mathcal{L}f_s)}{\Gamma(f_s)^{1-\frac{q}{2}}}\right)
\end{align*}
where we made use of the diffusion property for the generator $\mathcal{L}$ and
set  
$2\bar{\Gamma}(f):= \mathcal{L}(f^2)-2f\mathcal{L}f =
\sum_{i=1}^m \vert X_if \vert^2$. In what follows, the variables $i,k$
in the sums run over the ranges $\{1,\dots m\}$ and $\{1,\dots,n\}$ respectively.
For the first term, we have 
\begin{align*}
\mathcal{L} \Gamma(f_s)&=\sum_{i,k}  X_i^2(Z_kf_s)^2-\beta \sum_{k}
D (Z_kf_s)^2 \\&=2 \sum_{i,k} \vert X_iZ_kf_s\vert^2 + 2\sum_{k}
(Z_kf_s)(\mathcal{L}Z_kf_s).
%+2 \sum_{i,k} (Z_kf_s)(X_i^2Z_kf_s) - 2\beta\sum_{k} (Z_kf_s)(DZ_kf_s).
\end{align*}
On the other hand, the second term can be estimated as follows:
\begin{align*}\bar{\Gamma}(\Gamma(f_s))&=
\sum_{i}\left( X_i \left(\sum_{k} (Z_kf_s)^2\right)\right)^2 %\\ &
=4 \sum_{i}\left( \sum_{k}(Z_kf_s)(X_iZ_kf_s)\right)^2 \\ 
& \le 4 \left( \sum_{k}  \vert Z_kf_s \vert^2 \right) \left(\sum_{i,k} \vert X_i
Z_k f_s \vert ^2  \right)
%\\ &
= 4 \Gamma(f_s) \left(\sum_{i,k} \vert X_i Z_k f_s \vert ^2  \right).
\end{align*}
Finally
\begin{align*}
\Gamma(f_s, \mathcal{L}f_s) &=  \sum_{i,k} (Z_kf_s)(Z_kX_i^2f_s) 
-  \beta  \sum_{k} (Z_kf_s)(Z_kDf_s) \\
&= \sum_{k} (Z_kf_s)(\mathcal{L}Z_k f_s) 
+ \sum_{k} (Z_kf_s)([Z_k,\mathcal{L}]f_s) , 
\end{align*}  
Combining the above, we obtain 
\begin{align*}
\frac{d}{ds}&P_{t-s}\Gamma( f_s)^{\frac{q}{2}}   \le \\
&P_{t-s} \left( -q \frac{\left( \sum_{i,k} \vert X_iZ_kf_s\vert^2
+ \sum_{i,k}(Z_kf_s)(\mathcal{L}Z_kf_s)
\right)}{\Gamma(f_s)^{1-\frac{q}{2}}}\right)\\
&+P_{t-s}\left( q \frac{\left(  \sum_{i,k}(Z_kf_s)(Z_k \mathcal{L} f_s) 
 \right)}{\Gamma(f_s)^{1-\frac{q}{2}}}\right)
%\\ +&
+ P_{t-s} \left( q(2-q) \frac{\sum_{i,k} \vert X_iZ_kf_s
\vert^2}{\Gamma(f_s)^{1-\frac{q}{2}} }\right)\\
=& P_{t-s} \left( \frac{q\left(\sum_{i,k}\left((Z_kf_s)[Z_k,X_i^2]f_s
+(1-q)\vert X_iZ_kf_s \vert^2\right)
-\beta\sum_k(Z_kf_s)[Z_k,D]f_s\right)}{\Gamma(f_s)^{1-\frac{q}{2}}}\right)
\\\le&P_{t-s} \left(\frac{q}{\Gamma(f_s)^{1-\frac{q}{2}}}
\sum_{i,k}(Z_kf_s)[Z_k,X_i^2]f_s-\beta
\lambda_* q \vert \nabla f_s \vert
^q-\frac{q(q-1)}{\Gamma(f_s)^{1-\frac{q}{2}}}\sum_{i,k}
\vert X_iZ_kf_s \vert ^2 \right) ,
\end{align*}
where we made use of assumption \eqref{Condition_onD} and
$\lambda_*=\inf_k\lambda_k$.
The fact that $q>1$ is crucial, since this makes the coefficient in front of the
last term 
in the above expression non-zero,  hence allowing us to use it to  control the mixed derivatives coming from the
first term as follows. Observe that using assumption
\eqref{CommutatorConstantCoefficents}, 
we can write 
\begin{align*}
\sum_{i,k}(Z_kf_s)([Z_k,X_i^2]f_s)&=\sum_{i,k}(Z_kf_s)\left([Z_k,X_i](X_if_s)+X_i[Z_k,X_i]f_s\right)
\\&=\sum_{i,k}(Z_kf_s)\left(\sum_{l}c_{kil}\left(Z_lX_if_s+X_iZ_lf_s\right)\right)  \\ 
&=\sum_{i,k} (Z_kf_s)\left( \sum_l c_{kil} (2X_iZ_lf_s+[Z_l,X_i]f_s)\right)
\\&= \sum_{i,k}(Z_kf_s)\left(\sum_l 2c_{kil}X_iZ_lf_s +\sum_{l,m}c_{kil}c_{lim}Z_mf_s\right)
\\ &\le 2\sum_{i,k,l}c_{kil}(Z_kf_s)(X_iZ_lf_s)  \\& \qquad+\sum_{i,k,l,m}\vert c_{kil}c_{lim} \vert \left(\frac{\vert
Z_k f_s \vert ^2+\vert Z_mf_s \vert^2}{2}\right),
\end{align*}
where in the last step we used the Cauchy-Schwarz inequality. Thus 
\begin{align*}
\frac{d}{ds}&P_{t-s}\Gamma( f_s)^{\frac{q}{2}}  \le \\ 
& P_{t-s} \left(\frac{q}{\Gamma(f_s)^{1-\frac{q}{2}}}\sum_{i,k,l,m}
\left(2c_{kil} (Z_kf_s)(X_iZ_lf_s)+\vert c_{kil}c_{lim}\vert\left(\frac{\vert
Z_k f_s \vert ^2+\vert Z_mf_s \vert^2}{2}\right) \right) \right)\\ 
&-P_{t-s} \left( \beta \lambda_* q\Gamma( f_s)^{\frac{q}{2}} +
\frac{q(q-1)}{\Gamma(f_s)^{1-\frac{q}{2}}}\sum_{i,k}
\vert X_iZ_kf_s \vert ^2\right)
\\  \le&P_{t-s} \left(
\frac{q}{\Gamma(f_s)^{1-\frac{q}{2}}}\sum_{i,k,l}2c_{ikl}(Z_kf_s)(X_iZ_lf_s)
- \frac{q(q-1)\sum_{i,k}
\vert X_iZ_kf_s \vert ^2 }{\Gamma(f_s)^{1-\frac{q}{2}}}+C_1 \Gamma( f_s)^{\frac{q}{2}} \right)
\end{align*}
where $C_1=q \left(\frac{1}{2}\max_{k} \sum_{i,l,m}\vert c_{kil}c_{lim}\vert+\frac{1}{2}\max_{m} \sum_{i,l,k}\vert c_{kil}c_{lim}\vert-\beta \lambda_*\right) $.  %The proof will thus be complete if we
%show that there exists a constant $C_{2} $ 
%such that  
Using the inequality
\begin{equation*}
 \frac{q}{\Gamma(f_s)^{1-\frac{q}{2}}}\sum_{i,k,l}2c_{ikl}(Z_kf_s)(X_iZ_lf_s) \le C_2 \Gamma( f_s)^{\frac{q}{2}} 
+ \frac{q(q-1)}{\Gamma(f_s)^{1-\frac{q}{2}}}\sum_{i,k}
\vert X_iZ_kf_s \vert ^2.
\end{equation*}
with $C_2= \frac{q}{q-1} \left(\max_{k}\sum_{i,l}c^2_{ikl}\right)$,
%Let $\varepsilon>0$. Using Young's inequality, the left-hand side can be
%bounded by 
%\begin{align*}
%\frac{q\sum_{i,k,l}2c_{ikl}(Z_kf_s)(X_iZ_lf_s)}{\Gamma(f_s)^{1-\frac{q}{2}}} 
%&\le\frac{q}{\Gamma(f_s)^{1-\frac{q}{2}}} \sum_{i,l} \left(
%2\varepsilon\left(\sum_k
%c_{ikl}Z_kf_s\right)^2+\frac{1}{2\varepsilon}\vert X_iZ_lf_s \vert ^2
%\right)\\
%& \le \frac{2\varepsilon q}{\Gamma(f_s)^{1-\frac{q}{2}}} \sum_{i,k,l} c_{ikl}^2
%\vert Z_kf_s \vert ^2 
%+\frac{q}{2\varepsilon\Gamma(f_s)^{1-\frac{q}{2}}} \sum_{i,l}
%\vert X_i Z_lf_s\vert^2 \\ & \le2\varepsilon q
%\left(\sup_{k}\sum_{i,l}c^2_{ikl}\right)\vert
%\nabla f_s \vert ^q 
%+\frac{q}{2\varepsilon\Gamma(f_s)^{1-\frac{q}{2}}}\sum_{i,l} \vert X_iZ_lf_s
%\vert^2
%\end{align*} 
%Choosing $\varepsilon=\frac{1}{2(q-1)}$, we obtain the desired estimate with 
%$C_2= \frac{q}{q-1} \left(\sup_{k}\sum_{i,l}c^2_{ikl}\right)$.
we get
$$
\frac{d}{ds} P_{t-s}\Gamma( f_s)^{\frac{q}{2}}   \le - \kappa' P_{t-s}\Gamma( f_s)^{\frac{q}{2}} 
$$
with $\kappa'=-C_1-C_2$ which is
strictly positive for 
$$\beta>\frac{1}{\lambda_*}\left(\frac{1}{2}\max_{k} \sum_{i,l,m}\vert c_{kil}c_{lim}\vert+\frac{1}{2}\max_{m} \sum_{i,l,k}\vert c_{kil}c_{lim}\vert+\frac{1}{q-1} \left(\max_{k}\sum_{i,l}c^2_{ikl}\right) \right).
$$ 
Integration of the above differential inequality ends the proof.
\end{proof} 

Summarising, the key idea of our estimates is contained in the assumption of
completeness of the set of fields
$\{Z_k\}$ in the sense that their commutators with the fields appearing in  the
generator do not give essentially new fields. In case of free nilpotent Lie
groups (\cite{G-G}) they can be chosen simply by taking all the fields generated by
the fields defining the generator. In some cases (as for example groups of rank
2, when the fields have linear coefficients) our procedure will work with the
usual square of the gradient form. Then our method applied to the complete gradient
$\{Z_k\}$ provides some other useful information on monotonicity of derivatives.

\section{Extension to infinite dimensions} \label{sec.3}
Let $\mathbf{\Lambda}\Subset \mathbb{Z}^d$ be a finite subset
of the $d$-dimensional lattice. For each $k \in \mathbb{Z}^d$, we consider isomorphic copies of the vector fields 
$Z_1,\dots,Z_n$ denoted by $Z_{k,1},\dots,Z_{k,n}$
(and similarly the isomorphic copies of $X_i=Z_i$, $i=1,\dots,m)$. As
in the finite dimensional case, let 
\begin{align*} 
\mathcal{L}_k^{(1)}&= \sum_{i=1}^m X_{k,i}^2-\beta
D_k+  
\sum_{i,i'=1}^mG_{ii'}X_{k,i}X_{k,i'}, \\  
\mathcal{L}_k^{(2)}&=\sum_{i=1}^m \alpha_{k,i} X_{k,i},
\end{align*}
where \(G=(G_{ij})\) is a constant matrix satisfying \eqref{Condition_onG} and  
\(\alpha_{k,i}=\alpha_{k,i}(\omega), \; \omega \in ({\mathbb{R}^N})^{\mathbb{Z}^d}. \) 
We  will work under the additional assumption that the range of interaction is
finite. In other words  $\alpha_{k,i} $ will depend only on coordinates around
\(k,  \) i.e.    \(Z_{k,i}\alpha_{j,r}=0\) whenever 
\( |k-j| > R\).
Here, the distance  of two points on the lattice is defined as
\( |k-j| = \sum_{l=1}^d\vert k_l-j_l \vert \), where \(k=(k_1,\dots
k_l)\) and \(j=(j_1,\dots j_l)\). For %short, we will write 
%\(d_{\mathbb{Z}^d}(k) \equiv d_{\mathbb{Z}^d}(k,0)\) and for
 \(k \in \mathbb{Z}^d,\) \(\Lambda \subset \mathbb{Z}^d\), we set
\(dist (k,\Lambda) \equiv \inf\{ |k - l|): l \in \Lambda\}\) . 
In addition, we will assume that the quantities \(\Vert
Z_{k,r}\alpha_{j,i}\Vert_\infty\) are uniformly bounded in
\(k,j \in {\mathbb{Z}^d}. \)   %For each function \(f\),
By \(\Lambda(f)\) we will denote a localisation set for a function \(f\),
meaning that \(f\) depends only on coordinates indexed by points in
\(\Lambda(f)\). 
\par We consider a Markov
semigroup $P_t^\mathbf{\Lambda}$, defined via its generator
\[\mathcal{L}_\mathbf{\Lambda}=\sum_{k \in \mathbb{Z}^d }
\mathcal{L}_k^{(1)} +
\sum_{k \in\mathbf{\Lambda}} \mathcal{L}_k ^{(2)} 
\] 
and we define 
$\Gamma_k= \sum_{r=1}^n \vert Z_{k,r}f \vert ^2, 
\Gamma_{\mathbf{\Lambda}}=\sum_{k \in {\mathbf{\Lambda}}}
\Gamma_k$ and 
\(\Gamma=\sum_{k \in \mathbb{Z}^d} \Gamma_k\). 
This definition is motivated by the fact that the generators 
\(\mathcal{L}_{\mathbf{\Lambda}}\) approximate, as \({\mathbf{\Lambda}}
\uparrow \mathbb{Z}^d,\) the infinite dimensional generator 
\[ \mathcal{L} =\sum_{k \in \mathbb{Z}^d} \left( \mathcal{L}_k^{(1)}+
\mathcal{L}_k^{(2)}\right)
\equiv \sum_{k \in\mathbb{Z}^d} \mathcal{L}_k
.\] 
This construction, the details of which  are presented below, allows us to approximate 
the infinite dimensional semigroup \((e^{t \mathcal{L}})_{t \ge 0}\) by 
%finite volume 
Markov semigroups, which are easier to study.    

\subsection{Strong Approximation Property} \label{subsec.3.1}
%In this setting, Theorem \ref{Result_Constant Coefficients} can be extended in
%the following way.
%\par 
%\begin{teo}\label{infdimgradbound} \label{thm2.1}
%There exists $\varsigma  \in \mathbb{R}$ such that for all smooth functions
%$f$  localised in ${\mathbf{\Lambda}}$ and
%$k \in \mathbb{R}^{\mathbb{Z}^d}$
%\begin{equation}\label{Result_infinite_dimension}
% \Gamma_k(P_t^{\mathbf{\Lambda}} f) \le e^{-\varsigma t}
%P_t^{\mathbf{\Lambda}} \Gamma_{\mathbf{\Lambda}}(f).
%\end{equation}
%Moreover, there exists $\beta_0\in(0,\infty)$ such that for all $\beta>\beta_0$
%we have $\varsigma>0$.
%\end{teo}
%\proof

Given a finite set $\mathbf{\Lambda}\subset\mathbb{Z}^d$, for a cylinder
function $f$ such that 
$\Lambda(f)\subset\mathbf{\Lambda}$  we introduce 
$f_s=P_s^{\mathbf{\Lambda}}f$ 
and start similarly as before  by
considering 
\begin{align*} 
\partial_s P_{t-s}^{\mathbf{\Lambda}}\Gamma _k (f_s)&=
P_{t-s}^{\mathbf{\Lambda}} \left(
-\mathcal{L}_{\mathbf{\Lambda}}\Gamma_k(f_s)
+2\Gamma_k(f_s,\mathcal{L}_{\mathbf{\Lambda}} f_s)\right)\\ &
=P_{t-s}^{\mathbf{\Lambda}} \sum_{r=1}^n
\left(-\mathcal{L}_{{\mathbf{\Lambda}}} \vert Z_{k,r}(f_s)\vert^2
+2(Z_{k,r}f_s)\mathcal{L}_{\mathbf{\Lambda}} Z_{k,r}f_s 
+2(Z_{k,r}f_s)[Z_{k,r},\mathcal{L}_{\mathbf{\Lambda}}]f_s \right)
\\ & = P_{t-s}^{\mathbf{\Lambda}} \sum_{r=1}^n \left(   
-2{\bar\Gamma}(Z_{k,r}f_s) +
2(Z_{k,r}f_s)[Z_{k,r},\mathcal{L}_{\mathbf{\Lambda}}]f_s\right)
\end{align*}
where \(\bar{\Gamma}(f)= \sum_{j \in \mathbb{Z}^d} \left( \sum_{i,i'=1}^m (G_{ii'}+\delta_{ii'})(X_{j,i}f)(X_{j,i'}f) \right)\). Next, we notice that   
\begin{align*} 
[Z_{k,r},\mathcal{L}_{{\mathbf{\Lambda}}}]&=\sum_{j \in \mathbb{Z}^d}
\sum_{i=1}^m\left[Z_{k,r}, X_{j,i}^2-\beta D_j\right] + \sum_{j \in
\mathbb{Z}^d} \sum_{i,i'=1}^m\left[Z_{k,r}, G_{ii'}X_{j,i}X_{j,i'}\right]\\
&\quad+\sum_{j \in 
{\mathbf{\Lambda}}}\sum_{i=1}^m[Z_{k,r},\alpha_{j,i}X_{j,i}]\\ &
=%added 
\sum_{i=1}^m \left([Z_{k,r},X_{k,i}^2-\beta D_k]+
\alpha_{k,i} [Z_{k,r}X_{k,i}] \right) \\ &\quad + \sum_{i,i'=1}^m[Z_{k,r},G_{ii'}X_{k,i}X_{k,i'}]+ \sum_{j \in
{\mathbf{\Lambda}}} \sum_{i=1}^m
(Z_{k,r} \alpha_{j,i}) X_{j,i},
\end{align*}
because \(Z_{k,r}\) and \(X_{j,i}\) commute when \(j \neq k\) for all \(r,i\). 
Combining the above, we arrive at 
\begin{align*}
\partial_s & P_{t-s}^{\mathbf{\Lambda}}\Gamma _k (f_s)= %\\  
%&  \;
2P_{t-s}^{\mathbf{\Lambda}}   \sum_{r=1}^n \left( - 
{\bar\Gamma}(Z_{k,r}f_s) 
+ \sum_{i=1}^m (Z_{k,r}f_s)[Z_{k,r},X_{k,i}^2-\beta D_k]f_{s} \right. \\
& \phantom{AAAAAAAA}\left. + \sum_{i,i'=1}^m(Z_{k,r}f_s)[Z_{k,r},G_{ii'}X_{k,i}X_{k,i'}]f_{s}+\sum_{i=1}^m  \alpha_{k,i}(Z_{k,r}f_s) [Z_{k,r}X_{k,i}]f_{s} \right) \\ 
& \phantom{AAAAAAAA} +2P_{t-s}^{\mathbf{\Lambda}} \sum_{j \in {\mathbf{\Lambda}}}
\sum_{i=1}^m \sum_{r=1}^n  (Z_{k,r}f_s)(Z_{k,r} \alpha_{j,i})
(X_{j,i}f_s) \\ 
  &\; \phantom{QQAAQ}\le 2(-\beta\lambda_* + \tilde C) P_{t-s}^{\mathbf{\Lambda}}  \Gamma _k (f_s) %\\
%\left\{2P_{t-s}^{\mathbf{\Lambda}}   \sum_{r=1}^n \sum_{i=1}^m
%\left(  -\vert X_{j,r} Z_{k,r }f_s \vert^2+ (Z_{k,r}f_s) [Z_{k,r},X_{k,i}^2-\beta D_j
%+\sum_{i'=1}^mG_{ii'}X_{k,i}X_{k,i'}]f_{s}+ \alpha_{k,i} (Z_{k,r}f_s) [Z_{k,r}X_{k,i}]f_{s} \right) \right\} \\ 
%& \; 
+ 2P_{t-s}^{\mathbf{\Lambda}} \sum_{j \in
{\mathbf{\Lambda}}} \sum_{i=1}^m \sum_{r=1}^n 
(Z_{k,r}f_s)(Z_{k,r} \alpha_{j,i}) (X_{j,i}f_s),
\end{align*}
with some constant $\tilde C$ dependent only on the structure constants $c_{kjl}$, $G_{ij}$ and $||\alpha_{k,i} ||$.
%The first term  was already studied in the previous section and can be bounded from above by 
%\(-\kappa P_{t-s}^{\mathbf{\Lambda}} \Gamma_k(f_s)  \) with
%\(\kappa\) as in \eqref{mconstant}. 
For the last sum, we use Young's inequality to get 
\begin{align*}
2\sum_{j \in
{\mathbf{\Lambda}}}\sum_{i=1}^m&\sum_{r=1}^n(Z_{k,r}f_s)(Z_{k,r}
\alpha_{j,i})(X_{j,i}f_s)
\le\\&\sum_{i=1}^m \sum_{r=1}^n\vert Z_{k,r}\alpha_{k,i}\vert \left( \vert Z_{k,r}f_s \vert ^2
+\vert X_{k,i}f_s\vert^2 \right)\\&+\sum_{j \in
{\mathbf{\Lambda}},j\ne k}\sum_{i=1}^m \sum_{r=1}^n \vert
Z_{k,r}\alpha_{j,i}\vert
\left( \vert Z_{k,r}f_s\vert ^2 +\vert X_{j,i}f_s \vert ^2 \right) \\  \le 
\;&A_k\Gamma_{k}(f_s) %+B_{k}\Gamma_{k}(f_s)
+\sum_{j \in
{\mathbf{\Lambda}},j\ne k} %\left( C_{j,k} %\Gamma_k(f_s)+
M_{k,j}
\Gamma_j(f_s)%\right)
\end{align*}
with 
\begin{align*}\label{infdimconstants}\nonumber 
A_{k}&= \max _{r=1,\dots,n} \sum_{i=1}^m \Vert
Z_{k,r}\alpha_{k,i}\Vert_{\infty} +  
%& B _{k}=  
\max _{i=1,\dots,m} \sum_{r=1}^n \Vert Z_{k,r}\alpha_{k,i}\Vert_{\infty}
\\&\quad+
 %\\ 
%  C_{j,k}&=  
\sum_{j\in \mathbf{\Lambda},j \neq k}\max _{r=1,\dots,n} \sum_{i=1}^m \Vert Z_{k,r}\alpha_{j,i}\Vert_{\infty} ,  \\
M_{k,j} &=  \max _{i=1,\dots,m} \left(\sum_{r=1}^n \Vert Z_{k,r}\alpha_{j,i} \Vert_{\infty} \right),
\end{align*} 
which are finite quantities by our assumptions. 
We therefore arrive at 
\begin{equation}\label{infdimgradientbd}
\partial_s
P_{t-s}^{\mathbf{\Lambda}}\Gamma_k (f_s) \le
%-(\kappa-A_{k}-B_{k}-C_{k})P_{t-s}^{\mathbf{\Lambda}} \Gamma_k
%(f_s)
- \bar \kappa P_{t-s}^{\mathbf{\Lambda}}  \Gamma _k (f_s)
+ 
\sum_{j \in {\mathbf{\Lambda}}, j \ne
k}M_{k,j} P_{t-s}^{\mathbf{\Lambda}} \Gamma_j(f_s)
\end{equation}
with
\[
 C\equiv \tilde C+ \sup_{k \in \mathbb{Z}^d} A_k \qquad \text{ and } \qquad \bar
\kappa\equiv 2(\beta\lambda_* - C). 
\]
\endproof
Solving this differential inequality, we obtain the following bound:
\begin{lemma}\label{lem2.1}
There exists constants $\bar\kappa\in \mathbb{R}$ and $M_{k,j}\in(0,\infty), \,
M_{k,j}\equiv 0$ for $|j-k|> R$, such that 
for any $\mathbf{\Lambda}\subset\mathbb{Z}^d$ and any smooth cylinder function $f$ with $\Lambda(f)\subset\mathbf{\Lambda}$, we have
\begin{equation}
\Gamma_k (P_{t}^{\mathbf{\Lambda}} f) \leq e^{-\bar\kappa t}P_{t}^{\mathbf{\Lambda}}\Gamma _k (f)
+ \sum_{j \in {\mathbf{\Lambda}}, j \ne k}M_{k,j} \int_0^t ds\, e^{-\bar\kappa (t-s)}\, P_{t-s}^{\mathbf{\Lambda}} \Gamma_j(P_{s}^{\mathbf{\Lambda}} f).
\end{equation}
\end{lemma}
\begin{rem}  One can use this lemma to get gradient bounds in 
$l_q, \, q\geq 1$ norms for vectors $\Gamma_k,\, k\in\mathbb{Z}^d$.\end{rem}
\begin{rem} For a matrix $ \hat{G}=((\hat{G}_{ii'}^{kk'})_{i,i'=1}^m)_{k,k' \in \mathbb{Z}^d}$ satisfying $\hat{G} ^* +I  \ge 0$ and $ \sum_{k, k' \in \mathbb{Z}^d} \sum_{i,i'=1}^m\vert \hat G_{ii'}^{kk'}\vert < \infty$,  it is possible to repeat the above argument for the generator given by\[  \hat{\mathcal{L}}_\mathbf{\Lambda}=\mathcal{L}_\mathbf{\Lambda}+\sum_{k,k' \in \mathbb{Z}^d} \sum_{i,i'=1}^m \hat{G}_{ii'}^{kk'}X_{k,i}X_{k',i'}.\] 
\end{rem}
\begin{prop}[Finite speed of propagation of
information]\label{finitespeedofpropagation} \label{prop2.2}
Let \(f \) be a smooth function and assume 
$\Lambda(f) \subset{\mathbf{\Lambda}} \Subset \mathbb{Z}^d$. For
$k \notin \Lambda(f)$,  we have
\begin{align*}
\Vert \Gamma_k(P_t^{\mathbf{\Lambda}} f)\Vert_\infty 
&\le e^{N_k \left( \log C-\log N_k  + 2 + \log t\right)+Ct}\sum_{j \in
\mathbb{Z}^d}\Vert \Gamma_jf \Vert_\infty,
\end{align*}
where \(N_k=\left[ \frac{dist(k,\Lambda(f))}{R} \right]\) and \(C>0\) is a
constant. 
Hence, for any \(\sigma>0\) 
there exists \(\tau>1\) such that if \(N_k\ge \tau t\)
$$ \Vert \Gamma_k(P_t^{\mathbf{\Lambda}} f)\Vert_\infty \le
e^{-\sigma t-\sigma N_k}\sum_{j \in \mathbb{Z}^d}\Vert \Gamma_jf
\Vert_\infty .$$
\end{prop}\proof
We argue similarly as in \cite{G-Z} (see also references given there). 
From Lemma \ref{lem2.1}, we have
\begin{align*}
%\begin{equation}
||\Gamma_k (P_{t}^{\mathbf{\Lambda}} f)||_\infty \leq e^{-\bar\kappa t}||\Gamma _k (f)||_\infty
+ \sum_{j \in {\mathbf{\Lambda}}, j \ne k}M_{k,j} \int_0^t ds\, e^{-\bar\kappa (t-s)}\, || \Gamma_j(P_{s}^{\mathbf{\Lambda}} f)||_\infty
%\end{equation}
\end{align*}
 with  
\( M_{kj}\equiv 0,\, |j-k|\geq R 
\) 
and \(\kappa\in(0,\infty)\).  This implies 
\begin{align*} 
\Vert \Gamma_k(P_{t}f) \Vert_\infty &\le \Vert\Gamma_kf
\Vert_\infty +\sum_{j \in {\mathbf{\Lambda}}} M_{kj} \int _0^t 
\Vert \Gamma_j(f_{s}) \Vert_\infty ds\\&=\sum_{j \in
{\mathbf{\Lambda}}} M_{kj} \int _0^t \Vert \Gamma_j(f_{s})
\Vert_\infty ds, 
\end{align*}
since \(k \notin \Lambda(f)\). We may iterate the above to get \[ \Vert \Gamma_k(P_tf)\Vert_\infty 
\le C_{}^{N_{k}}\frac{t^{N_k}}{N_k!}e^{Ct}\sum_{j \in \mathbb{Z}^d} \Vert \Gamma_jf \Vert_\infty,
\]
with some constant \(C>0\) and  \(N_k=\left[ \frac{dist(k,\Lambda(f))}{R} \right]\).   Since 
\(N_{k}! > e^{N_k\log N_{k}-2N_{k}}, \) 
\begin{align*}
\Vert \Gamma_k(P_tf)\Vert_\infty &\le e^{N_k \left( \log C-\log N_k  +2+\log t\right)+Ct}\sum_{j \in \mathbb{Z}^d}
\Vert \Gamma_jf \Vert_\infty. 
\end{align*}
Now, given \(\sigma>0,\) we may choose \(\tau \geq 1\) large enough  so that 
$ \log\frac{C}{\tau}+2+\frac{C}{\tau}\le -2\sigma $. 
If \(N_k \ge \tau t,\)  we then get  
\begin{align*}
N_k \left( \log C-\log N_k  +2+\log t\right)+Ct
& \le N_k \left(\log\frac{C}{\tau} +2+\frac{C}{\tau}\right) \\ 
& \le N_k (-2\sigma) \le -\sigma N_k -\sigma t,
\end{align*}
as required.
\endproof
\begin{teo}\label{lamlimit} \label{thm2.3}
For any \(t>0\)  and any continuous function %all cylinder smooth functions
\(f\) %localised on a finite subset \(\Lambda(f) \Subset \mathbb{Z}^d\) 
the following limit exists in the uniform
norm
\[\lim_{{\mathbf{\Lambda}} \uparrow \mathbb{Z}^d}
P_t^{{\mathbf{\Lambda}}}f=:P_tf\ \]
and defines a Markov semigroup.
\end{teo} 
\proof 
It is sufficient to prove the existence of the limit for smooth cylinder
functions.
To this end pick \({\mathbf{\Lambda}}_1,{\mathbf{\Lambda}}_2\Subset
\mathbb{Z}^d\) such that \(\Lambda(f)\subset
{\mathbf{\Lambda}}_1\subset{\mathbf{\Lambda}}_2\) and choose a
sequence 
\({\mathbf{\Lambda}}^{(n)}\) such that
\({\mathbf{\Lambda}}^{(0)}={\mathbf{\Lambda}}_1\),
\({\mathbf{\Lambda}}^{(\mathcal{N})}={\mathbf{\Lambda}}_2\) 
and     
\({\mathbf{\Lambda}}^{(n+1)}\setminus
{\mathbf{\Lambda}}^{(n)}=\{j_n\}\) is a singleton for any
\(n=0,\dots,\mathcal{N}-1\). 
Using the Fundamental Theorem of Calculus and the fact that
\(P_{t}^{{\mathbf{\Lambda}}^{(n)}}\) is contractive, we have 
\begin{align*}
\left\Vert
P_t^{{\mathbf{\Lambda}}_2}f-P_t^{{\mathbf{\Lambda}}_1}
f\right\Vert_\infty 
&\le \sum_{n=0}^{\mathcal{N}-1} \left\Vert
P_t^{{\mathbf{\Lambda}}^{(n)}}f-P_t^{{\mathbf{\Lambda}}^{(n+1)}}
f\right\Vert_\infty\\ 
& \le \sum_{n=0}^{\mathcal{N}-1} \int_0^t \left\Vert
P_{t-s}^{{\mathbf{\Lambda}}^{(n)}} 
(\mathcal{L}_{{\mathbf{\Lambda}}^{(n)}}-\mathcal{L}_{{\mathbf{\Lambda}}
^{(n+1)}})P_s^{{\mathbf{\Lambda}}^{(n+1)}}f\right\Vert_\infty ds 
\\ 
& \le\sum_{n=0}^{\mathcal{N}-1} \int_0^t 
\left\Vert  
(\mathcal{L}_{{\mathbf{\Lambda}}^{(n)}}-\mathcal{L}_{{\mathbf{\Lambda}}
^{(n+1)}})P_s^{{\mathbf{\Lambda}}^{(n+1)}}f\right\Vert_\infty ds \\ 
& \le\sum_{n=0}^{\mathcal{N}-1} \int_0^t \left\Vert 
\sum_{i=1}^m
\alpha_{j_n,i}X_{j_n,i}P_s^{{\mathbf{\Lambda}}^{(n+1)}}
f\right\Vert_\infty ds
\\ 
&\le \sum_{n=0}^{\mathcal{N}-1} \int_0^t 
\left\Vert \sum_{i=1}^m \alpha_{j_n,i}^2 \right\Vert ^\frac{1}{2} _\infty 
\left\Vert \sum_{i=1}^m \left\vert X_{j_n,i}P_s
^{{\mathbf{\Lambda}}^{(n+1)} }f\right\vert^2
\right\Vert^\frac{1}{2}_\infty ds.
\\ & \le m \sum_{n=0}^{\mathcal{N}-1} 
\int_0^t \Vert \alpha_{j_n,i}\Vert_\infty
\sqrt{\Vert\Gamma_{j_n}(P_s ^{{\mathbf{\Lambda}}^{(n+1)} }f)
\Vert_\infty} ds
\end{align*}
Let \(\sigma>0\). Since \(j_n \notin {\mathbf{\Lambda}}_1\),
we can apply Lemma \ref{finitespeedofpropagation} to conclude
that  
\begin{align}\nonumber
\left\Vert
P_t^{{\mathbf{\Lambda}}_2}f-P_t^{{\mathbf{\Lambda}}_1}
f\right\Vert_\infty & \le m
\sum_{n=0}^{\mathcal{N}-1}\Vert \alpha_{j_n,i}\Vert_\infty \int_0^t e^{-\sigma
s-\sigma N_{j_n}} \sqrt{ \sum_{j \in \mathbb{Z}^d}\Vert \Gamma_j
f\Vert_{\infty}}ds \\ \label{expdecay}& \le m\mathcal{N}e^{-\sigma\overline{N}\
} \frac{1-e^{-\sigma t}}{\sigma}\max_{n=1,\dots,\mathcal{N}-1}\Vert
\alpha_{j_n,i}\Vert_\infty \sqrt{ \sum_{j \in \mathbb{Z}^d}\Vert \Gamma_j
f\Vert_\infty}ds 
\end{align}
provided that  \(N_{j_n}= \left[ \frac{dist(j_{n},\Lambda(f))}{R}
\right] \ge \overline{N}\ge \tau t\) for some \(\tau>1\) large enough, where 
\(\overline{N} =\left[
\frac{dist({\mathbf{\Lambda}}_1,\Lambda(f))}{R}
\right]\). We have therefore established that  if
\(({\mathbf{\Lambda}}_n\Subset
\mathbb{Z}^d)_{n = 0}^\infty\) is
a sequence such that \({\mathbf{\Lambda}}_n
\uparrow \mathbb{Z}^d\) as \(n \rightarrow \infty, \) then
\((P_t^{{\mathbf{\Lambda}}_n}f)_{n=0}^\infty\) is  a Cauchy
sequence.
\endproof 

\subsection{Existence of a limit measure} \label{limitmeasureinfdim}
\label{subsec2.2} %\\
Let \(\zeta \in \mathbb{N}\) be such that 
\( \sum_{k \in \mathbb{Z}^d}(1+ |k|)^{-\zeta} < \infty\). 
For \(\mathscr{K} \in \mathbb{N}\), define sets \[\Omega_\mathscr{K}
=\left\{\omega\in ({\mathbb{R}^N})^{\mathbb{Z}^d}: 
\sum_{k \in \mathbb{Z}^d}(1+ |k|)^{-\zeta} d(\omega_k)<\mathscr{K}\right\}\]
and let 
$$\Omega:=\cup_{\mathscr{K} \in \mathbb{N}}(
\Omega_{\mathscr{K}})=\left\{\omega\in ({\mathbb{R}^N})^{\mathbb{Z}^d}: 
\sum_{k \in \mathbb{Z}^d}(1 + |k|)^{-\zeta}d(\omega_k)<\infty\right\}.
$$
For \(j \in \mathbb{Z}^d\) and \(\omega \in ({\mathbb{R}^N})^{\mathbb{Z}^d}\) we
consider the (semi-)distance \(d_j(\omega)\equiv d(\omega_j)\) (recall that we write \(d(x,0)=d(x)\) where \(d\) is a metric on $\mathbb{R}^N$). The
corresponding cut-off \(\rho_j\) then satisfies,  (similarly as in Lemma
\ref{Ptdbound}), 
\[P_t^{\mathbf{\Lambda}} \sum_{j \in
{\mathbf{\Lambda}}}\rho_j \le K_{\mathbf{\Lambda}}\] for some
constant \(K_{\mathbf{\Lambda}}>0\) and all \(t>0\).  
If we define
\(\Upsilon^{\mathbf{\Lambda}}_{L}=\{\sum_{j \in
{\mathbf{\Lambda}} } \rho_j\le L\}\), 
arguing as in
Section \ref{limitmeasurefinitedim}, we can extract a convergent subsequence 
\(P_{t_k}^{\mathbf{\Lambda}}\) such that for all bounded
continuous \(f\) and \(\omega \in
\Omega\) we have \(P_{t_k}^{\mathbf{\Lambda}} f(\omega)
\rightarrow \nu _{{\mathbf{\Lambda}}, \omega} (f).\)

\subsection{Ergodicity of the semigroup} 
   
%
%%%%%%%%%%%%%%%%%%%%%%%%%%%%%%%%%%%%%%%%%%%%%%%%%%%%

By Section \ref{limitmeasureinfdim} and  Theorem \ref{lamlimit} we have  that  for
\(\omega \in \Omega\) there exists a measure \(\nu_\omega\) such that \[ P_{t_k}
^{\mathbf{\Lambda}} f(\omega) \rightarrow \nu_\omega(f) \] as
\(k \rightarrow \infty \) and
\({\mathbf{\Lambda}} \uparrow \mathbb{Z}^d\). Moreover, by
Markov's inequality, for
all~\(\omega \in \Omega\) 

\[\nu_{\omega} (\Omega_\mathscr{K}) \ge 1-\frac{1}{\mathscr{K}}
\sup_{k \in \mathbb{Z}^d}\left(\int %_{\mathbb{R}^{\mathbb{Z}^d}}
d(x_k)\nu_{\omega}(dx)\right)\sum_{k \in \mathbb{Z}^d}(1+|k|)^{-\zeta}
\] 
and
thus \(\nu_{\omega}(\Omega)=1\). 
We will show the following result.
%that   
%\begin{equation}\label{ptomega}
%\vert P_{t}(\omega)-P_{t}(\tilde{\omega})\vert \rightarrow 0 
%\quad \text{ as }\quad t \rightarrow \infty,
%\end{equation}
%for  \(\omega, \tilde{\omega} \in \Omega\). \par 
 
\begin{teo} \label{thm2.4}
There exists \(t_0>0\) such that for \(t>t_0 \),   bounded
smooth cylinder function \(f\)  and any \( \omega, \tilde{\omega} \in \Omega, \)
$$ \vert P_t
f(\omega)-P_tf(\tilde{\omega})\vert \le { \mathcal{C} }(f,\omega,
\tilde{\omega})e^{-\varpi t},$$ where \(\varpi>0\) is a constant and \(  {
\mathcal{C} }(f,\omega, \tilde{\omega}) \) depends only on \(f, \omega\) and
\(\tilde{\omega}\).  
\end{teo}

\begin{proof} We choose \({\mathbf{\Lambda}}=
{\mathbf{\Lambda}}(t)\) such
that \(\diam ({\mathbf{\Lambda}}) =\varkappa t\) for some
\(\varkappa>0\) to be determined
later, and order the elements of  ${\mathbf{\Lambda}}$
lexicographically. 
For \(\omega,\tilde{\omega} \in \Omega\) we can choose a suitable sequence
\( (\omega^k)_{k \in \mathbb{Z}^d}\) that interpolates between \(\omega\) and
\(\tilde{\omega}\) and such that each element  differs from the previous one
only in single coordinate.  Moreover for all \({\mathbf{\Lambda}}
\Subset \mathbb{Z}^d\),
\begin{align*} 
\vert P_t f(\omega)-P_tf(\tilde{\omega})\vert \le& \vert P_t
f(\omega)-P^ {\mathbf{\Lambda}}_t f(\omega) \vert
+\vert P_t
^{{\mathbf{\Lambda}}}f(\omega)-P^{\mathbf{\Lambda}}_tf(\tilde{
\omega})\vert \\ &  + \vert P_t f(\tilde{\omega})
-P^{\mathbf{\Lambda}}_t f(\tilde{\omega}) \vert.  
\end{align*}
By the proof of Theorem \ref{lamlimit} and the fact that 
\(\diam ({\mathbf{\Lambda}}) =\varkappa t\), we can find \(T>0\) 
such that  \( t>T/\varkappa\) implies 
\[\left\vert P_t f(\omega)-P^ {\mathbf{\Lambda}}_t f(\omega)
\right\vert+\left\vert P_t f(\tilde{\omega})
-P^ {\mathbf{\Lambda}}_t f(\tilde{\omega}) \right\vert \le
\mathcal{C}_1(f,\omega,\tilde{\omega}) e^{-\theta t/2},\]
where \(\theta\in(0,\infty)\) and \(\mathcal{C}_1(f,\omega,\tilde{\omega}) \)
is a finite  constant depending on the cylinder function \(f\) and configurations \(\omega, \tilde{\omega}\). We also
have \[\left\vert 
P_t^{\mathbf{\Lambda}} f(\omega)-P_t ^{\mathbf{\Lambda}}
f(\tilde{\omega})\right\vert \le \sum_{k \in
{\mathbf{\Lambda}}^{R}}\left\vert  P_t^{\mathbf{\Lambda}}
f(\omega^{k+1})-P_t^{\mathbf{\Lambda}} f(\omega^k)\right\vert,
\] 
with
\({\mathbf{\Lambda}}^R=\{k \in \mathbb{Z}^d:
dist(k, {\mathbf{\Lambda}})\le R\}\) where \(R\) is the range of interaction.
Let \(\gamma:[0,t_{k}]\rightarrow \Omega\) be an admissible path connecting \(\omega^k\) to \({\omega^{k+1}}\), 
such that \(\dot{\gamma}_s=1  \) (recall that \(\omega^k\) and \(\omega^{k+1}\)
differ only in the \(k^{th}\) coordinate, 
so \(t_{k}=d(\omega_k, \tilde{\omega}_k\))).
%added
 The differential inequality \eqref{infdimgradientbd} implies that 
 \begin{equation*}
 \partial_s P_{t-s}^\mathbf{\Lambda}\Gamma_k(P_s ^\mathbf{\Lambda}f) \le -( \bar{\kappa }-\max_{j \in \mathbb{Z}^d }M_{k,j})P_{t-s}^\mathbf{\Lambda}\Gamma_\mathbf{\Lambda}(P_s ^\mathbf{\Lambda}f)
 \end{equation*}
 (recall that $M_{k,j}\equiv0$ when $\vert j - k \vert >R$), which after integration gives  
\begin{equation*}
\Gamma_k(P_t^\mathbf{\Lambda} f)\le e^{-\varsigma t} P_t^\mathbf{\Lambda}(\Gamma_\mathbf{\Lambda}f)
\end{equation*}
with some $\varsigma \in \mathbb{R}$ which is positive for large $\beta$ and can be made independent of \(k\) by our assumption that the quantities \(\Vert
Z_{k,r}\alpha_{j,i}\Vert_\infty\) are uniformly bounded in
\(k,j \in {\mathbb{Z}^d} \).  This observation together with contractivity  property of \(P^{\mathbf{\Lambda}}_{t}\) imply
\begin{align} \label{p_tlambda}
\sum _{k\in{\mathbf{\Lambda}}^{R}}
&\left\vert P_t^{\mathbf{\Lambda}}
f(\omega^{k+1})- P_t^{\mathbf{\Lambda}} f(\omega^k)\right\vert
%\nonumber \\
%&
\le\sum _{k\in{\mathbf{\Lambda}}^{R}}  
\int_0^{t_{k}} 
\sqrt{\Gamma_k (P_t ^{\mathbf{\Lambda}} f(\gamma_s))}  ds \nonumber \\
&\le\sum _{k \in {\mathbf{\Lambda}}^{R}}
(d(\omega_{k}) + d(\tilde{\omega}_k))
\left\Vert\Gamma_{k}(P_t^{\mathbf{\Lambda}} f ) \right\Vert^{\frac{1}{2}}_\infty
%\nonumber \\ 
%& 
\le  \sum_{k\in{\mathbf{\Lambda}}^R} 
(d(\omega_{k})+ d(\tilde{\omega}_k))e^{-\frac{\varsigma t}{2}} 
\left\Vert{ \Gamma_{{\mathbf{\Lambda}}}(f)
}\right\Vert_\infty^{\frac{1}{2}}  \nonumber \\ 
& \le e^{-\frac{\varsigma t}{2}} \left( \sum_{k \in
\mathbb{Z}^d}\Vert{ \Gamma_{k}(f)
}\Vert_\infty\right)^{\frac{1}{2}}  
\sum_{k \in {\mathbf{\Lambda}}^R} 
(d(\omega_{k})+d(\tilde{\omega}_k))  \\  
& \le e^{-\frac{\varsigma t}{2}}  \left( \sum_{k \in
\mathbb{Z}^d} \Vert{ \Gamma_{k}(f)
}\Vert_\infty\right)^{\frac{1}{2}} 
(C_{\omega} + C_{\tilde{\omega}}) (1+\varkappa t)^\zeta \nonumber
\end{align}  
using that \(\vert k \vert \le \varkappa t\) since \(k \in
{\mathbf{\Lambda}}\), with
\(C_{\omega}\equiv\sum_{k \in \mathbb{R}^{\mathbb{Z}^d}}(1+\vert k \vert
)^{-\zeta}d(\omega_k) \) which is finite since \(\omega \in \Omega\) 
(and similarly for $C_{\tilde{\omega}}$). Hence there exists a constant
\(\mathcal{C}_2(f,\omega, \tilde{\omega})\) such that \[
\left\vert  P_t^{\mathbf{\Lambda}}
f(\omega)-P_t ^{\mathbf{\Lambda}} f(\tilde{\omega})\right\vert  
\le
\mathcal{C}_2(f,\omega,\tilde{\omega})(1+\varkappa t)^\zeta e^{- \varsigma
t/2}.\]  \par Combining the above 
%and choosing  \(\varkappa \in (0,\varsigma/2\zeta)\) 
we conclude that  \(t>T/\varkappa\equiv t_0\)
implies
\begin{align*}\left\vert P_tf(\omega)-P_t f(\tilde{\omega})\right\vert &
\le { \mathcal{C} }(f,\omega,\tilde{\omega})e^{-\varpi t} ,  
\end{align*}
for some constants \(\varpi >0\) and \( { \mathcal{C}
}(f,\omega,\tilde{\omega})\) 
depending only on the cylinder function \(f\) 
and configurations \(\omega, \tilde{\omega}\).
\end{proof}

\begin{rem} We note that in fact the estimate \eqref{p_tlambda} is sufficiently strong to include the configurations with exponential growth for which $\sum_k e^{-\gamma |k|}d(\omega_k)< \infty$ with any $\gamma < \varsigma /2$ which is much more than a set of measure one.\end{rem}

%\endproof

\subsection{Properties of the Invariant Measure} \label{subsec2.4}
Recall the following representation of a covariance
\begin{equation} \label{eq2.4.1}
P_tf^2 - (P_tf)^2 = 2 \int_0^t\partial_s P_s \bar\Gamma(P_{t-s}f)ds 
\end{equation}
Since $\bar\Gamma \leq \Gamma$, if we have the following bound
$$
\Gamma(P_\tau f)\leq e^{-\kappa\tau}P_\tau\Gamma(f)
$$ 
then \eqref{eq2.4.1} implies the following result

\begin{teo} \label{thm2.4.1}
Under the conditions on the generator $\mathcal{L}$
there exists $\beta_0\in(0,\infty)$ such that for all $\beta>\beta$, any
differentiable function $f$, at any $t>0$ 
\begin{equation} \label{eq2.4.2}
P_tf^2 - (P_tf)^2 \leq \frac2\kappa (1-e^{-\kappa t})\, P_t \Gamma(f). 
\end{equation}
Hence the unique $P_t$-invariant measure $\nu$ satisfies
$$
\nu(f-\nu f)^2\leq \frac2\kappa\, \nu \Gamma(f). 
$$
\end{teo}

We mention that by abstract arguments, (see e.g. \cite{G-Z}, Exercise 2.9, and
references therein), the Poincar\'e type inequality \eqref{eq2.4.2} implies a
uniform in $t>0$ exponential bound 
\begin{equation} \label{eq2.4.3}
P_t e^{\delta f} < Const\, e^{\frac{\delta^2}\kappa \Gamma(f)}e^{\delta P_t f}
\end{equation}
provided
\begin{equation} \label{eq2.4.4}
\frac{\delta^2}\kappa||\Gamma(f)||_\infty \leq 1.
\end{equation}
Application of this property
yields the following exponential bound result:
\begin{coroll} \label{thm2.4.2} Under the condition of the Theorem 
\ref{thm2.4.1}
the invariant measure $\nu$ satisfies the following exponential bound
$$
\nu \left( e^{\delta f}\right) < Const\, e^{\frac{\delta^2}\kappa
\Gamma(f)}e^{\delta  \nu(f)}
$$ 
for any function $f$ satisfying \eqref{eq2.4.2} and for which $\nu f$ is well
defined. 
\end{coroll}
{\textbf{Remark}}{\textsl{ An interesting question arises, which was also a part
of motivation to our
work, whether the measure $\nu$ can satisfy stronger coercive inequalities
as for example Log-Sobolev inequality.
The known strategy of \cite{B-E} to obtain log-Sobolev requires bounds with
$\Gamma_1$ and and unfortunately fails in cases of interest to us in this paper. }\\

{\textbf{Remark}}{\textsl{
Note that knowing a bit of regularity one can slightly optimize \eqref{eq2.4.1}
as follows. First we use 
$$
P_tf^2 - (P_tf)^2 \leq \int_0^{t-\varepsilon} ds 2 P_s\Gamma (P_{t-s}f) + 
\int_{t-\varepsilon}^t ds 2 P_s\Gamma(P_{t-s}f)
$$
If one would have the following regularity estimate
$$
\Gamma (P_\varepsilon f) \leq \bar c(\varepsilon)\Gamma_1(f)
$$
then we get
$$
\int_0^{t-\varepsilon} ds 2 P_s\Gamma (P_{t-s}f) = \int_0^{t-\varepsilon} ds 2
P_s\Gamma (P_{t-\varepsilon-s}P_\varepsilon f) 
\leq 2 \bar c(\varepsilon) \int_0^{t-\varepsilon} ds
e^{-\kappa(t-\varepsilon-s)} P_{t-\varepsilon}\Gamma_1(f) 
$$
$$
\leq \frac2\kappa  \bar c(\varepsilon) (1-e^{-\kappa (t-\varepsilon)})
P_{t-\varepsilon}\Gamma_1(f) 
\leq \frac2\kappa  \bar c(\varepsilon) P_{t-\varepsilon}\Gamma_1(f) 
$$
On the other hand we have
$$
\int_{t-\varepsilon}^t ds 2 P_s\Gamma(P_{t-s}f) \leq
\frac2\kappa(1-e^{-\kappa\varepsilon}) P_{t}\Gamma (f) 
\leq  \varepsilon\frac2\kappa \varepsilon P_{t}\Gamma (f) 
$$
Hence, with $\gamma(\varepsilon)$ given as an inverse function of
$c(\varepsilon) \equiv \frac2\kappa \bar c(\frac\kappa2 \varepsilon)$, we obtain
$$
P_tf^2 - (P_tf)^2 \leq \varepsilon 
P_{t-\varepsilon}\Gamma_1(f) + \gamma(\varepsilon) P_t\Gamma(f)   
$$
After passing with time to infinity we obtain
$$
\nu (f - \nu f)^2 \leq \varepsilon \nu \Gamma_1(f) +
\gamma(\varepsilon)\nu\Gamma(f) 
$$
which after optimisation with respect to the free parameter $\varepsilon$ implies a
generalised Nash type inequality.  
}}

\end{document}